\documentclass[10pt]{article}
 \makeatletter
\newfont{\footsc}{cmcsc10 at 8truept}
\newfont{\footbf}{cmbx10 at 8truept}
\newfont{\footrm}{cmr10 at 10truept}
\renewcommand{\ps@plain}{
	\renewcommand{\@oddfoot}{\footsc {\footbf }  \footrm\thepage}}

\makeatother  \leftmargin=25mm
\usepackage{color}
\usepackage[centertags]{amsmath}
\usepackage{tikz}
\usepackage{pgfplots}
\usepackage{amsfonts}
\usepackage{amsthm}
\usepackage{newlfont}
\usepackage{graphicx}
\usepackage{epstopdf}
\usepackage{subfigure}
\usepackage[utf8]{inputenc}
\usepackage{indentfirst}
\usepackage{latexsym}
\usepackage{float}
\usepackage{cite}
\usepackage{CJK}
\usepackage{array}
\usepackage{amsfonts}
\usepackage{mathrsfs}
\usepackage{amssymb}
\usepackage{mathtools}
\usepackage[ruled,linesnumbered]{algorithm2e}
\usepackage{float}     
\usepackage[T1]{fontenc}
\usepackage{authblk}
\usepackage{setspace}
\usepackage{bm}
\usepackage{tikz}
\usetikzlibrary{calc,angles}
\usepackage{circuitikz}
\usepackage{caption}
\usepackage{subfigure}
\usepackage{framed}
\usepackage{centerlastline}

\usepackage[colorlinks,
urlcolor=blue,
linkcolor=black,
anchorcolor=black,
citecolor=black]{hyperref}

\newtheorem{thm}{Theorem}[section]

\newtheorem{lem}[thm]{Lemma}
\newtheorem{prop}[thm]{Proposition}

\theoremstyle{definition}
\newtheorem{defn}[thm]{Definition}

\newtheorem{rem}[thm]{Remark}

\title{On the Unimodular Isomorphism Problem of Convex Lattice Polytopes\footnotemark[2]\footnotetext[2]{This work is supported by the National Natural Science Foundation of China (NSFC12226006, NSFC11921001) and the Natural Key Research and Development Program of China (2018YFA0704701).}}
\date{}
\author{Qiuyue Liu\textsuperscript{1} and Zhanyuan Cai\textsuperscript{1,}\thanks{Corresponding author. E-mail address: 18976257573@163.com. Cellphone number: +86-18976257573} \\  
	{\small$^{1}$ Center for Applied Mathematics, Tianjin University, Tianjin, 300072, China}\\
}
\begin{document}
	\maketitle
	\begin{abstract}

This paper studies the \emph{unimodular isomorphism problem} (UIP) of convex lattice polytopes: given two convex lattice polytopes $P$ and $P'$, decide whether there exists a unimodular affine transformation mapping $P$ to $P'$. We show that UIP is graph isomorphism hard, while the polytope
congruence problem and the combinatorial  polytope isomorphism problem (Akutsu, 1998; Kaibel, Schwartz, 2003) were shown to be graph isomorphism complete, and both the lattice isomorphism problem ( $\mathrm{Sikiri\acute{c}}$, $\mathrm{Sch\ddot{u}rmann}$, Vallentin, 2009) and the projective/affine polytope isomorphism problem (Kaibel, Schwartz, 2003) were shown to be graph isomorphism hard. Furthermore, inspired by protocols for lattice (non-) isomorphism (Ducas, van Woerden, 2022; Haviv, Regev, 2014), we present a statistical zero-knowledge proof system for unimodular isomorphism of lattice polytopes. Finally, we propose an algorithm that given two lattice polytopes computes all unimodular affine transformations mapping one polytope to another and, in particular, decides UIP.


	\end{abstract}

	\maketitle
	\section{Introduction}	
	A \emph{convex lattice polytope} in $\mathbb{E}^{n}$ is the convex hull of a finite subset of the integer lattice $\mathbb{Z}^n$. Equivalently, it is a convex polytope, all vertices of which are in $\mathbb{Z}^n$. Convex polytopes possess a rich structure arising from the interaction of algebraic, convex, analytic, and combinatorial properties, which gives rise to many  combinatorial and computational challenges. For general references on polytopes and lattice polytopes, we refer to \cite{Barvinok}, \cite{Ziegler},\cite{Gritzmann} and \cite{Gruber}.
	
	Two convex lattice polytopes $P, P'\subset \mathbb{Z}^{n}$ are said to be $\mathbb{Z}$-\emph{isomorphic} or \emph{unimodularly isomorphic} if there is an affine map $\mathcal{U}:\mathbb{R}^{n} \to \mathbb{R}^{n}$ with $\mathcal{U}(\mathbb{Z}^{n})=\mathbb{Z}^{n}$ and $\mathcal{U}(P)=P'$. This is clearly an equivalence relation, and equivalent lattice polytopes have the same volume and the same number of lattice points. One of the most fundamental problems about this equivalence relation is Arnold's problem,  where given the positive integer $m$ we study the number $v(n,m)$ of different classes of $n$-dimensional convex lattice polytopes with volume $m/n!$. In \cite{Arnold}, \cite{I}, \cite{Konyagin}, \cite{Pach} and \cite{Vershik}, upper and lower bounds of $v(n,m)$ have been achieved. More information about Arnold's problem can be found in \cite{I} and \cite{Liu}.  Furthermore, bounds for the numbers of different classes of $n$-dimensional convex lattice polytopes with the fixed number of lattice points  have been achieved in \cite{Zong} and \cite{Liu}. Blanco and Santos \cite{Blanco} described an algorithm to classify all convex lattice 3-polytopes of width larger
	than one and with a given number of lattice points. In 2021, Balletti\cite{Gabriele} also developed an algorithm for the complete enumeration of equivalence classes of $n$-dimensional lattice polytopes having (normalized) volume at most $m$.
	
	This paper is concerned with the \emph{unimodular isomorphism problem} (UIP) of convex lattice polytopes, namely, deciding whether two given convex lattice polytopes $P, P'\subset \mathbb{Z}^{n}$ are unimodularly isomorphic. The analogous problems were studied in \cite{Akutsu} and \cite{Kaibel}. Akutsu \cite{Akutsu} showed that the polytope congruence problem is graph isomorphism complete. Kaibel and Schwartz \cite{Kaibel} showed that the combinatorial isomorphism problem of convex polytopes is graph isomorphism complete and the projective$/$affine isomorphism problem is graph isomorphism hard, and gave a polynomial time algorithm for the combinatorial polytope isomorphism problem in bounded dimensions.
	
	Deciding whether two given combinatorial or algebraic structures are isomorphic is a notorious question in the theory of computing. One of its notable cases is the \emph{graph isomorphism problem} (GIP), in which given two graphs $G$ and $G'$, one has to decide whether there is a bijection $\varphi$ between their node sets such that $\{\varphi(v), \varphi(w)\}$ is an edge of $G'$ if and only if $\{v, w\}$ is an edge of $G$. If the nodes or edges of the two graphs are labeled, and if we consider only isomorphisms that preserve the labels, then the corresponding decision problem is called the \emph{labeled graph isomorphism problem}. GIP remains one of the few natural problems in NP that is neither known to be NP-complete nor known to be polynomial-time solvable. In a breakthrough, Babai \cite{Babai} proved that GIP is solvable in quasi-polynomial time $n^{polylog(n)}:= n^{(logn)^{\mathcal{O}(1)}}$, where $n$ denotes the number of nodes of the input graphs. This was the first improvement on the worst-case running time of a general graph isomorphism algorithm over the $2^{\mathcal{O}(\sqrt{nlogn})}$ bound established in \cite{Luks}. 
	
	Another well-known isomorphism problem is the \emph{lattice isomorphism problem} (LIP), in which, given two lattices $\Lambda$ and $\Lambda'$, the goal is to decide whether there exists an orthogonal linear transformation mapping $\Lambda$ to $\Lambda'$. Interestingly, it was shown in \cite{Dutour} that the isomorphism problem on lattices is at least as hard as that on graphs. In addition, Haviv and Regev \cite{Regev} proposed a statistical zero-knowledge proof system for  lattice non-isomorphism.  In 2022, Ducas and van Woerden \cite{Ducas} further built on their techniques to propose a statistical zero-knowledge proof system for lattice isomorphism, which directly implies an identification scheme  based on the search version of LIP (search-LIP).
	
	The structure of this paper is as follows. In Section 2, we present the relevant notations and concepts. In Section 3, we reduce the GIP to the UIP in polynomial time and thus the UIP is graph isomorphism hard. In Section 4,  thanks to Gaussian sampling \cite{Gentry} and the techniques developed in \cite{Ducas}, we define an average-case distribution for the UIP and propose a statistical zero-knowledge proof system for unimodular isomorphism. Finally, we describe an algorithm which given  $n$-dimensional convex lattice polytopes $P, P'\subset \mathbb{Z}^{n}$ with $d$ vertices computes all unimodular affine transformations mapping $P$ to $P'$ and, in particular, decides UIP.
	\section{Preliminaries}
	\subsection{Notation}
	The vector $\bm{x}$ is denoted in bold and should be interpreted as a column vector. We denote the isomorphism relation by $\cong$. Suppose $C\in \mathbb{R}^{n\times n'}$ is a matrix and $\bm{x}\in \mathbb{R}^{n}$, then we use $C+\bm{x}$ to denote the matrix obtained by adding the column vector $\bm{x}$ to each column vector of the matrix $C$.  
	We consider the full dimensional convex lattice polytopes $P$ with $d$ vertices in $n$-dimensional space. Let $\mathcal{V}(P)$ denote the vertex set of $P$, and let $|\mathcal{V}(P)|$ denote the number of vertices. We denote the unimodular isomorphism class of $P$ by $[P]$. Algorithmically, we can represent $P$ by an $n\times d$ matrix $V_P$ whose columns are given by the vertex vectors. Obviously, $V_P$ is uniquely defined only up to permutations $\sigma\in S_{d}$ of the columns. 
	
	Let $\mathrm{Aut}(P):=\{\mathcal{A}=(A,\bm{a})\in GL_n(\mathbb{Z})\times \mathbb{Z}^{n}:\mathcal{A}(P)=AP+\bm{a}=P\}$ denote the automorphism group of convex lattice polytope $P$. We denote the set of all unimodular matrices in $\mathrm{Aut}(P)$ and the set of all lattice translations in $\mathrm{Aut}(P)$ by $\mathrm{Aut_u}(P)$ and $\mathrm{Aut_t}(P)$, respectively. Obviously, $\mathrm{Aut}(P)$ is finite. If there exists $\mathcal{U}=(U,\bm{Z})$ such that $P'=UP+\bm{Z}$, the set of all isomorphisms from $P$ to $P'$ is given by $\{\mathcal{UA}=(UA,U\bm{a}+\bm{Z})\in GL_n(\mathbb{Z})\times \mathbb{Z}^{n}:\mathcal{A}\in \mathrm{Aut}(P)\}$.  
	\subsection{Unimodular Isomorphism Problem, Lattices and Quadratic Forms}
\noindent \emph{Unimodular isomorphism problem.} We formulate the unimodular isomorphism problem as a decision problem.

\begin{defn}[\textbf{lattice polytope version}] 
Given $n$-dimensional convex lattice polytopes $P,P'\subset \mathbb{Z}^{n}$, decide whether there exists a unimodular matrix $U\in GL_{n}(\mathbb{Z})$ and an integer vector $\bm{Z}\in \mathbb{Z}^{n}$ such that $P'=UP+\bm{Z}$.
\end{defn}
	Algorithmically convex lattice polytopes $P, P'\subset \mathbb{Z}^{n}$ are represented by vertex matrices $V_{P}, V_{P'}\in \mathbb{Z}^{n\times d}$, and if $P'=UP+\bm{Z}$, then $UV_{P}+\bm{Z}$ is a vertex matrix of $P'$. If $UV_{P}+\bm{Z}=V_{P'}$, then we can easily compute $U$ and $\bm{Z}$. However, in general $UV_{P}+\bm{Z}$ will only be equal to $V_{P'}$ up to some permutation matrix. More specifically, we give the following definition.
	\begin{defn}[\textbf{algorithm version}]
		Given the vertex matrices $V_{P}, V_{P'}\in \mathbb{Z}^{n\times d}$ of $n$-dimensional convex lattice polytopes $P,P'\subset \mathbb{Z}^{n}$, decide whether there exists a unimodular matrix $U\in GL_{n}(\mathbb{Z})$, an integer vector $\bm{Z}\in \mathbb{Z}^{n}$ and a $d\times d$ permutation matrix $M$ such that $V_{P'}=UV_{P}M+\bm{Z}$.
	\end{defn}
	\begin{defn}
		Let $\mathcal{V}(P)$ be the set of vertices of a convex polytope $P$. Then the \emph{vertex average} of $P$ is the point  
		\begin{align*}
			b_{P}:=\frac{1}{|\mathcal{V}(P)|}\sum_{\bm{v}\in \mathcal{V}(P)}\bm{v}.
		\end{align*}
	\end{defn} 
	For any full dimensional polytope $P\subset \mathbb{R}^{n}$, we can easily obtain the following proposition.
	\begin{prop}\label{ercixing}
		For any $n$-dimensional polytope $P\subset \mathbb{R}^{n}$, let $Q=(V_{P}-b_P)(V_{P}-b_P)^{t}$, then $Q$ is an $n \times n$  positive definite real symmetric matrix. 
	\end{prop}
	Obviously, for the given $n$-dimensional polytope $P\subset \mathbb{R}^{n}$, the quadratic form $Q=(V_{P}-b_P)(V_{P}-b_P)^{t}$ is uniquely determined and we can easily obtain the following
	proposition. 
	\begin{prop}
		Two $n$-dimensional convex lattice polytopes $P, P'\subset \mathbb{Z}^{n}$ are unimodularly isomorphic if and only if $V_{P'}-b_{P'}=U(V_{P}-b_{P})M$, where $U\in GL_{n}(\mathbb{Z})$ and $M$ is a $d\times d$ permutation matrix.
	\end{prop}
	Thus, if $n$-dimensional convex lattice polytopes $P,P'\subset \mathbb{Z}^{n}$ are unimodularly isomorphic, then
	\begin{align}\label{ecxtg} Q'=(V_{P'}-b_{P'})(V_{P'}-b_{P'})^{t}=U(V_{P}-b_{P})MM^{t}(V_{P}-b_{P})^{t}U^{t}=UQU^{t}.
	\end{align}
	
\vskip 0.2cm
\noindent \emph{Lattices and quadratic forms.} A (\emph{geometric}) \emph{lattice} $\Lambda$ in $\mathbb{E}^{n}$ is the system of all integer linear combinations of $n$
linearly independent vectors $\bm{b}_1,...,\bm{b}_n \in \mathbb{E}^{n}$, 
\begin{align*}
	\Lambda=\{a_1 \bm{b}_1+\cdots+a_n \bm{b}_n: a_i \in \mathbb{Z}\}.
\end{align*}
The $n$-tuple $\{\bm{b}_1,...,\bm{b}_n\}$ is called a \emph{basis} of $\Lambda$. Let B be the $n\times n$ matrix whose $i$-th column is $\bm{b}_i$. We define a \emph{quadratic form} as a positive definite real symmetric matrix. 

For a lattice basis $B\in GL_n(\mathbb{R})$, the Gram matrix $Q=B^{t}B$ naturally induces a quadratic form.  
Every quadratic form $Q\in S^{>0}_{n}$ admits a unique upper-triangular matrix $B_{Q}$ such that $Q=B^{t}_{Q}B_{Q}$ via Cholesky decomposition, where $ S^{>0}_{n}$ denotes the set of all $n\times n$ positive definite real symmetric matrices. In the quadratic
form framework, lattice vectors $B\bm{x}\in \mathbb{R}^n$ are represented by their integral basis coefficients $\bm{x} \in \mathbb{Z}^n$. The inner product with respect to a quadratic form is naturally given by $\langle \bm{x}, \bm{y} \rangle_Q := \bm{x}^{t}Q\bm{y}$, and the norm by $\Vert \bm{x}\Vert^{2}_{Q}:=\bm{x}^{t}Q\bm{x}$. Note that this perfectly coincides with the geometry between the original lattice vectors. We denote the ball of radius $r$ by $\mathcal{B}_Q(r):= \{\bm{x} \in \mathbb{R}^n: \Vert \bm{x} \Vert_{Q} \leqslant r\}$. Translating the lattice definition, the \emph{ first successive minimum} $\lambda_1(Q)$
is given by 
\begin{equation*}
	\lambda_{1}(Q) := \mathop{\mathrm{min}}\limits_{\bm{x}\in \mathbb{Z}^{n}\setminus\{\bm{0}\}} \Vert \bm{x} \Vert_{Q},
\end{equation*}
and more generally, the \emph{$i$-th successive minimum}
 $\lambda_i(Q)$ is the smallest $r>0$ such that $\{\bm{x}\in \mathbb{Z}^n : \Vert \bm{x} \Vert_{Q} \leqslant r\}$ spans a space of dimension at least $i$.
 
\subsection{Discrete Gaussians and Sampling}
Discrete Gaussian sampling has been fundamental to lattice-based cryptography, enabling the generation of short or near-center lattice vectors while ensuring no leakage of secret key information \cite{Gentry}. 
 We introduce the relevant concepts and define the discrete Gaussian distribution $D_{P, s, \bm{c}}$ in the language of quadratic forms. 
\vskip 0.2cm
\noindent \emph{Distribution.}  For any vectors $\bm{c},\bm{x}\in \mathbb{R}^{n}$ and any $s>0$, let
\begin{align*}
	\rho_{s,\bm{c}}(\bm{x})=e^{-\pi \parallel (\bm{x}-\bm{c})/s \parallel^{2}}
\end{align*}
be a \emph{Gaussian function} centered in $\bm{c}$ and scaled by a factor of $s$. The total measure associated to $\rho_{s,\bm{c}}$ is $\int_{\bm{x}\in \mathbb{R}^{n}}\rho_{s,\bm{c}}(\bm{x})d\bm{x}=s^n$.  Therefore, we can define the (\emph{continuous}) \emph{Gaussian distribution} around $\bm{c}$ with parameter $s$ by its probability density function
\begin{align*}
	\forall \bm{x}\in \mathbb{R}^{n},\ D_{s,\bm{c}}(\bm{x})=\frac{\rho_{s,\bm{c}}(\bm{x})}{s^n}.
\end{align*}
The discrete Gaussian distribution is obtained by restricting the continuous
Gaussian distribution to a discrete lattice. Then, for any vector $\bm{c}$, real $s>0$, and lattice $\Lambda$, we define the \emph{discrete Gaussian distribution} $D_{\Lambda,s,\bm{c}}$ over $\Lambda$ by
\begin{align}\label{lisangaosi}
	\forall \bm{x}\in \Lambda, \ D_{\Lambda,s,\bm{c}}(\bm{x})=\frac{D_{s,\bm{c}}(\bm{x})}{D_{s,\bm{c}}(\Lambda)}=\frac{\rho_{s,\bm{c}}(\bm{x})}{\rho_{s,\bm{c}}(\Lambda)},
\end{align}
where $\rho_{s,\bm{c}}(\Lambda)=\sum_{\bm{x}\in \Lambda}\rho_{s,\bm{c}}(\bm{x})$.

For any $n$-dimensional polytope $P\subset \mathbb{R}^{n}$, $Q := (V_{P}-b_{P})(V_{P}-b_{P})^{t}$ is the $n\times  n$ positive definite real symmetric matrix.
We define the Gaussian function $\rho_{P,s,\bm{c}}$ on $\mathbb{R}^{n}$ with center  $\bm{c}\in \mathbb{R}^{n}$ and parameter $s > 0$ by
\begin{align*}
	\forall \ \bm{x}\in \mathbb{R}^{n},\ \rho_{P,s,\bm{c}}(\bm{x})=\rho_{Q,s,\bm{c}}(\bm{x}):=\exp(-\pi \| \bm{x}-\bm{c}\|^{2}_{Q}/s^{2}),
\end{align*}
where $\| \bm{x}-\bm{c}\|^{2}_{Q}=(\bm{x}-\bm{c})^{t}Q(\bm{x}-\bm{c})$.
In the quadratic form setting, the discrete lattice will always be $\mathbb{Z}^{n}$, but with the geometry induced by the quadratic form. For any $n$-dimensional polytope $P\subset \mathbb{R}^{n}$, we define the discrete Gaussian distribution $D_{P,s,\bm{c}}$ with center $\bm{c}\in \mathbb{R}^{n}$ and parameter $s>0$ by
\begin{equation}\label{Dps}
	\mathop{\Pr} \limits_{X\thicksim D_{P,s,\bm{c}}} [X=\bm{x}] =\mathop{\Pr}\limits_{X \thicksim D_{Q,s,\bm{c}}}  [ X =\bm{x}]: = \frac{\rho_{Q,s,\bm{c}}(\bm{x})}{\rho_{Q,s,\bm{c}}(\mathbb{Z}^{n})}\ \mathrm{if}\ \bm{x}\in \mathbb{Z}^{n}, \mathrm{and}\ 0\ \mathrm{otherwise}.
\end{equation}
If the center $\bm{c}$ is not denoted, we have $\bm{c}=\bm{0}$.
 
\vskip 0.2cm
\noindent \emph{Gaussian sampling.} For many practical purposes, it is necessary to actually sample (close to) the discrete Gaussian distribution in an efficient manner. Gentry et al. \cite{Gentry} gave an efficient procedure that samples within negligible statistical distance of any (not too
narrow) discrete Gaussian distribution. Brakerski et al. \cite{Brakerski} modified this sampler so that the output is distributed exactly as a discrete Gaussian. To simplify later proofs, we use an exact sampling algorithm by Brakerski et al. \cite{Brakerski} and follow the formulation in \cite{Ducas}.
\begin{lem}[\cite{Brakerski},\cite{Ducas}]\label{caiyang} There is a polynomial-time algorithm $\bm{\mathrm{DiscreteSample}}(Q,s,\bm{c})$ that given a quadratic form $Q\in S^{>0}_{n}$, center $\bm{c}\in \mathbb{R}^{n}$, and a parameter $s\geq \parallel B^{*}_{Q}\parallel \cdot \sqrt{ln(2n+4)/\pi}$, returns a sample distributed as $D_{Q,s,\bm{c}}$.
\end{lem}
Here, $B^{*}_{Q}$ denotes the Gram-Schmidt orthogonalization of $B_{Q}$, and $\parallel B^{*}_{Q}\parallel$ is the length of the longest vector
in it. 

Through the above process, for the given $n$-dimensional convex lattice polytope $P\subset \mathbb{Z}^{n}$ represented by the vertex matrix $V_{P}$, we can obtain the quadratic form  $Q=(V_{P}-b_P)(V_{P}-b_{P})^{t}\in S^{>0}_{n}$ and the  discrete Gaussian distribution  $D_{P,s,\bm{c}}=D_{Q,s,\bm{c}}$, and sample from the distribution by Lemma \ref{caiyang}. 

\section{The Complexity of Unimodular Isomorphism Problem}
In this section, we give the complexity of unimodular isomorphism problem by Karp reducing graph isomorphism problem to it. Specifically, a decision problem $C$ is \emph{Karp reducible} to another decision problem $D$, if there is a polynomial time algorithm which constructs from an instance $I$ of $C$ an instance $J$ of $D$ with the property that the answer for $J$ is "yes" if and only if the answer for $I$ is "yes". Often a decision problem $\varPi$ is called \emph{graph isomorphism hard} if the graph isomorphism problem is Karp reducible to $\varPi$. 

Then, the result in this section is based on a polynomial time construction from graphs to convex lattice polytopes. Let $G = (V, E)$ be a graph with $n = |V|$ nodes, and $P(G)$ denote the convex lattice polytope constructed from $G$.

\paragraph{Step 1:} Choose a bijection $\sigma$ from the node set $V= \{v_1, \ldots, v_n\}$ to  $n$ vertices of the $(n-1)$-dimensional lattice simplex $\Delta_{n-1} := \text{conv}(\bm{e}_1, \ldots, \bm{e}_n)$, where $\bm{e}_i$ is the $i$-th $n$-dimensional unit vector and $\sigma(v_i)=\bm{e}_i$, $i=1,...,n$.

\paragraph{Step 2:} Construct the convex lattice polytope 
$P(G)$ by taking the convex hull of the origin, standard basis vectors, and edge-associated vectors. Specifically, associate each edge $\{v_i, v_j\}$ of $G$ with \(\bm{e}_i + \bm{e}_j\), yielding the set of vectors $\mathcal{S}(P(G))=\{\bm{0}, \bm{e}_1,..., \bm{e}_n, \bm{e}_{i_1}+\bm{e}_{j_1},...,\bm{e}_{i_t}+\bm{e}_{j_t}\}$, where $t=|E|$ and $i_1, ...,i_t, j_1, ..., j_t \in \{1, 2, ..., n\}$. Then, $P(G)=\text{conv}(\mathcal{S}(P(G)))$. 
Obviously, $P(G)$ is an $n$-dimensional convex lattice polytope.

\begin{rem} Obviously, there exists a polynomial-time algorithm that computes the finite point set  $\mathcal{S}(P(G))$ of \(P(G)\) from \(G\).
\end{rem}
The important property of $P(G)$ is that it encodes the entire structure of $G$. The subset $\{\bm{e}_1,...,\bm{e}_n\}$ of $\mathcal{S}(P(G))$ encodes the node set of graph $G$, while the subset $\{\bm{e}_{i_1}+\bm{e}_{j_1},...,\bm{e}_{i_t}+\bm{e}_{j_t}\}$ encodes the edge set of graph $G$. 

\begin{thm} 
  There is a Krap reduction	of the graph isomorphism problem to the unimodular isomorphism problem.
\end{thm}
\begin{proof} Let \(G=(V_G,E_G)\) and \(G'=(V_{G'},E_{G'})\) be graphs and let $n_g=|V_G|$ and $n_{g'}=|V_{G'}|$. 
	
(1). Suppose $E_G=E_{G'}=\varnothing$. Through the above polynomial time construction, we can get $n_g$-dimensional lattice polytope $P(G)=\mathrm{conv}(\bm{0},\bm{e}_1,...,\bm{e}_{n_g})$ and $n_{g'}$-dimensional lattice polytope $P(G')=\mathrm{conv}(\bm{0},\bm{e}_1,...,\bm{e}_{n_{g'}})$. \(G \cong G'\) if and only if $n_g=n_{g'}$.  If \(G \cong G'\), then $P(G)=P(G')=\mathrm{conv}(\bm{0},\bm{e}_1,...,\bm{e}_{n_g})$ and $P(G)$, $P(G')$ are unimodularly isomorphic. 
 If $P(G)\cong P(G')$, then $n_g=n_{g'}$ and \(G \cong G'\).

(2). Suppose $E_G \neq \varnothing$ and $E_{G'}\neq \varnothing$. We modify $G$ and $G'$ by adding vertices $v_g$ and $v_{g'}$ to $V_G$ and $V_{G'}$ respectively, where $v_g$ is adjacent to all vertices in $V_G$ and $v_{g'}$ is adjacent to all vertices in $V_{G'}$. We call the new graphs $H$ and $H'$, and  they have the following properties.
	\begin{itemize}
	\item \(H \cong H'\) if and only if $ G \cong G'$.
	\item  Graphs $H$ and $H'$ are connected and have triangular cycles.
	\item Both $H$ and $H'$ contain a node that is adjacent to all other nodes.
\end{itemize}

  Let $n=|V_H|$ and $n'=|V_{H'}|$. Through the above polynomial time construction, we can get $n$-dimensional lattice polytope $P(H)=\mathrm{conv}(\bm{0},\bm{e}_1,...,\bm{e}_{n},\bm{e}_{i_1}+\bm{e}_{j_1},...,\bm{e}_{i_t}+\bm{e}_{j_t})$ and $n'$-dimensional lattice polytope $P(H')=\mathrm{ conv}(\bm{0},\bm{e}_1,...,\bm{e}_{n'},\bm{e}_{i^{'}_{1}}+\bm{e}_{j^{'}_{1}},...,\bm{e}_{i^{'}_{r}}+\bm{e}_{j^{'}_{r}})$.  
	
	If \(G \cong G'\), then \(H \cong H'\), $n=n'$ and $t=r$. From the previous construction, the graph isomorphism $\varphi : V_{H} \to V_{H'}$ induces a bijection $\phi$ from the set $\{\bm{e}_1,...,\bm{e}_n\}$ to itself, which is represented by an $n\times n$ permutation matrix $M$.
	The edge $\{v_i, v_j\}$ in $H$ is represented by the vector $\bm{e}_i+\bm{e}_j$. Since $\varphi$
	preserves adjacency, $\{\varphi(v_i), \varphi (v_j)\}$ is an edge in $H'$, corresponding to $\bm{e}_{\varphi(i)}+\bm{e}_{\varphi(j)}=M(\bm{e}_i+\bm{e}_j)$. Thus, $\phi$ is also a bijection between the sets $\{\bm{e}_{i_1}+\bm{e}_{j_1},...,\bm{e}_{i_t}+\bm{e}_{j_t}\}$ and $\{\bm{e}_{i^{'}_{1}}+\bm{e}_{j^{'}_{1}},...,\bm{e}_{i^{'}_{r}}+\bm{e}_{j^{'}_{r}}\}$,  implying that $P(H')=MP(H)$.

	Next we give the proof of another direction. If the two instance lattice polytopes $P(H)$ and $P(H')$ constructed are unimodularly isomorphic, then $n=\mathrm{dim}(P(H))=\mathrm{dim}(P(H'))=n'$, and there is a unimodular matrix $U\in GL_n(\mathbb{Z})$ and an integer vector $\bm{Z}\in \mathbb{Z}^{n}$ such that $P(H')=UP(H)+\bm{Z}$. Thus, $t=r$ and we have 
	
	\begin{equation}\label{duiying}
		\begin{aligned}
			\mathcal{S}(P(H'))&=\{\bm{0},\bm{e}_1,...,\bm{e}_n,\bm{e}_{i^{'}_{1}}+\bm{e}_{j^{'}_{1}},...,\bm{e}_{i^{'}_{t}}+\bm{e}_{j^{'}_{t}}\}\\
			&=\{\bm{Z},U\bm{e}_1+\bm{Z},...,U\bm{e}_n+\bm{Z},U(\bm{e}_{i_1}+\bm{e}_{j_1})+\bm{Z},...,U(\bm{e}_{i_t}+\bm{e}_{j_t})+\bm{Z}\}.
		\end{aligned} 
	\end{equation}

	Let $U=(\bm{u}_1,...,\bm{u}_n)$, then $U\bm{e}_i=\bm{u}_i$ and (\ref{duiying}) becomes
	\begin{equation}\label{duiying1}
		\begin{aligned}
			\mathcal{S}(P(H'))&=\{\bm{0},\bm{e}_1,...,\bm{e}_n,\bm{e}_{i^{'}_{1}}+\bm{e}_{j^{'}_{1}},...,\bm{e}_{i^{'}_{t}}+\bm{e}_{j^{'}_{t}}\}\\
			&=\{\bm{Z},\bm{u}_1+\bm{Z},...,\bm{u}_n+\bm{Z},\bm{u}_{i_1}+\bm{u}_{j_1}+\bm{Z},...,\bm{u}_{i_t}+\bm{u}_{j_t}+\bm{Z}\}.
		\end{aligned} 
	\end{equation}
 Subtracting $\bm{Z}$ from the elements of the two sets in equation (\ref{duiying1}) yields
	\begin{equation}\label{duiying2}
		\begin{aligned}
			&\{-\bm{Z},\bm{e}_1-\bm{Z},...,\bm{e}_n-\bm{Z},\bm{e}_{i^{'}_{1}}+\bm{e}_{j^{'}_{1}}-\bm{Z},...,\bm{e}_{i^{'}_{t}}+\bm{e}_{j^{'}_{t}}-\bm{Z}\}=\\
			&\{\bm{0},\bm{u}_1,...,\bm{u}_n,\bm{u}_{i_1}+\bm{u}_{j_1},...,\bm{u}_{i_t}+\bm{u}_{j_t}\}.
		\end{aligned}
	\end{equation}

Next, we perform a case-by-case analysis of the elements in  $\{\bm{u}_1,...,\bm{u}_n\}$
and  $\{\bm{u}_{i_1}+\bm{u}_{j_1},...,\bm{u}_{i_t}+\bm{u}_{j_t}\}$. Before proceeding, we first examine the set $\mathcal{S}=\{\bm{u}_1,...,\bm{u}_n,\bm{u}_{i_1}+\bm{u}_{j_1},...,\bm{u}_{i_t}+\bm{u}_{j_t}\}$. Since $\bm{u}_1,...,\bm{u}_n$ are the column vectors of a unimodular matrix, all elements of $\mathcal{S}$ are non-zero vectors.  Additionally, each element in $\{\bm{u}_{i_1}+\bm{u}_{j_1},...,\bm{u}_{i_t}+\bm{u}_{j_t}\}$ can be expressed as the sum of two elements in $\mathcal{S}$, whereas the elements in $\{\bm{u}_1,...,\bm{u}_n\}$
cannot be represented as such a sum. \\

\vskip 0.1cm
	\textbf{case 1:} If $-\bm{Z}=\bm{0}$, we have 
	\begin{equation}\label{duiying3}
		\begin{aligned}
			\{\bm{0},\bm{e}_1,...,\bm{e}_n,\bm{e}_{i^{'}_{1}}+\bm{e}_{j^{'}_{1}},...,\bm{e}_{i^{'}_{t}}+\bm{e}_{j^{'}_{t}}\}=
			\{\bm{0},\bm{u}_1,...,\bm{u}_n,\bm{u}_{i_1}+\bm{u}_{j_1},...,\bm{u}_{i_t}+\bm{u}_{j_t}\}.
		\end{aligned}
	\end{equation}
	
 We claim $\{\bm{u}_1,...,\bm{u}_n\}=\{\bm{e}_1,...,\bm{e}_n\}$ and $\{\bm{u}_{i_1}+\bm{u}_{j_1},...,\bm{u}_{i_t}+\bm{u}_{j_t}\}=\{\bm{e}_{i^{'}_{1}}+\bm{e}_{j^{'}_{1}},...,\bm{e}_{i^{'}_{t}}+\bm{e}_{j^{'}_{t}}\}$. Otherwise, there will be some $\bm{u}_i=\bm{e}_{i^{'}_{k}}+\bm{e}_{j^{'}_{k}}$ and some $\bm{e}_j=\bm{u}_{i_m}+\bm{u}_{j_m}$, where $i,j\in\{1,...,n\}$, $k,m \in\{1,...,t\}$. According to (\ref{duiying3}), $\{\bm{u}_1,...,\bm{u}_n\}$ is a subset of the set $\{\bm{e}_1,...,\bm{e}_n,\bm{e}_{i^{'}_{1}}+\bm{e}_{j^{'}_{1}},...,\bm{e}_{i^{'}_{t}}+\bm{e}_{j^{'}_{t}}\}$. This indicates that the vector $\bm{e}_j$ can be expressed as the sum of two vectors in $\{\bm{e}_1,...,\bm{e}_n,\bm{e}_{i^{'}_{1}}+\bm{e}_{j^{'}_{1}},...,\bm{e}_{i^{'}_{t}}+\bm{e}_{j^{'}_{t}}\}$, which is impossible. Therefore, in this case,  it follows that $\{\bm{u}_1,...,\bm{u}_n\}=\{\bm{e}_1,...,\bm{e}_n\}$ and $\{\bm{u}_{i_1}+\bm{u}_{j_1},...,\bm{u}_{i_t}+\bm{u}_{j_t}\}=\{\bm{e}_{i^{'}_{1}}+\bm{e}_{j^{'}_{1}},...,\bm{e}_{i^{'}_{t}}+\bm{e}_{j^{'}_{t}}\}$.\\
 
	\vskip 0.1cm
 \textbf{case 2:} If $\bm{e}_i-\bm{Z}=\bm{0}$, $i\in \{1,...,n\}$, then (\ref{duiying2}) becomes 
	\begin{equation}\label{case2}
		\begin{aligned}
			&\{-\bm{e}_i,\bm{e}_1-\bm{e}_i,...,\bm{0},...,\bm{e}_n-\bm{e}_i,\bm{e}_{k_1},...,\bm{e}_{k_d},\bm{e}_{p_{1}}+\bm{e}_{q_{1}}-\bm{e}_i,...,\bm{e}_{p_{s}}+\bm{e}_{q_{s}}-\bm{e}_i\}=\\
			&\{\bm{0},\bm{u}_1,...,\bm{u}_n,\bm{u}_{i_1}+\bm{u}_{j_1},...,\bm{u}_{i_t}+\bm{u}_{j_t}\},   
		\end{aligned}
	\end{equation}
	where $d+s=t$ and the vectors \( \bm{e}_{k_1}, \ldots, \bm{e}_{k_d} \) correspond to nodes adjacent to the node associated with \( \bm{e}_i \). Let \( \bm{e}_{z_1}, \ldots, \bm{e}_{z_m} \) correspond to the nodes non-adjacent to this node, so that $d+m=n-1$. Since $H'$ is connected, we have $1\leqslant d \leqslant n-1$ and $0\leqslant m \leqslant n-2$. 

	If $m=0$, we can obtain
	\begin{equation}\label{case21}
		\begin{aligned}
			&\{-\bm{e}_i,\bm{e}_{k_1}-\bm{e}_i,...,\bm{e}_{k_{n-1}}-\bm{e}_i, \bm{e}_{k_1},...,\bm{e}_{k_{n-1}},\bm{e}_{p_{1}}+\bm{e}_{q_{1}}-\bm{e}_i,...,\bm{e}_{p_{s}}+\bm{e}_{q_{s}}-\bm{e}_i\}\\
			&=\{\bm{u}_1,...,\bm{u}_n,\bm{u}_{i_1}+\bm{u}_{j_1},...,\bm{u}_{i_t}+\bm{u}_{j_t}\}.   
		\end{aligned}
	\end{equation}
 For the left-hand side, the elements $-\bm{e}_i,\bm{e}_{k_1},...,\bm{e}_{k_{n-1}}$ cannot be expressed as the sum of two elements in this set, so $\{-\bm{e}_i, \bm{e}_{k_1},...,\bm{e}_{k_{n-1}}\}\subseteq \{\bm{u}_1,...,\bm{u}_n\}$. Since both sets have cardinality $n$, it follows that
	$\{\bm{u}_1,...,\bm{u}_n\}=\{-\bm{e}_i,\bm{e}_{k_1},...,\bm{e}_{k_{n-1}}\}$. By the presence of triangular cycles in $H'$, the left-hand side of (\ref{case21}) contains $\bm{e}_{k_x}+\bm{e}_{k_y}-\bm{e}_i$ ($x,y\in \{1,...,n-1\}$), which must lie in $\{\bm{u}_{i_1}+\bm{u}_{j_1},...,\bm{u}_{i_t}+\bm{u}_{j_t}\}$. However, this element cannot be decomposed into two elements of $\{\bm{u}_1,...,\bm{u}_n\}$, contradicting (\ref{case21}). 

If $1\leqslant m \leqslant n-2$, we have
	\begin{equation}\label{case22}
		\begin{aligned}
			&\{-\bm{e}_i,\bm{e}_{z_1}-\bm{e}_i,...,\bm{e}_{z_m}-\bm{e}_i,\bm{e}_{k_1}-\bm{e}_i,...,\bm{e}_{k_d}-\bm{e}_i,\bm{e}_{k_1},...,\bm{e}_{k_d},\bm{e}_{p_{1}}+\bm{e}_{q_{1}}-\bm{e}_i,...,\bm{e}_{p_{s}}+\bm{e}_{q_{s}}-\bm{e}_i\}\\
			&=\{\bm{u}_1,...,\bm{u}_n,\bm{u}_{i_1}+\bm{u}_{j_1},...,\bm{u}_{i_t}+\bm{u}_{j_t}\}.   
		\end{aligned}
	\end{equation}
 Here, the elements $-\bm{e}_i,\bm{e}_{z_1}-\bm{e}_i,...,\bm{e}_{z_m}-\bm{e}_i,\bm{e}_{k_1},...,\bm{e}_{k_d}$ cannot be expressed as the sum of two elements in the set, implying $\{-\bm{e}_i,\bm{e}_{z_1}-\bm{e}_i,...,\bm{e}_{z_m}-\bm{e}_i,\bm{e}_{k_1},...,\bm{e}_{k_d}\}\subseteq \{\bm{u}_1,...,\bm{u}_n\}$. With both sets of size $n$, we get
  $\{\bm{u}_1,...,\bm{u}_n\}=\{-\bm{e}_i,\bm{e}_{z_1}-\bm{e}_i,...,\bm{e}_{z_m}-\bm{e}_i,\bm{e}_{k_1},...,\bm{e}_{k_d}\}$. Triangular cycles in $H'$ introduce elements like $\bm{e}_{k_x}+\bm{e}_{k_y}-\bm{e}_i$ ( $x,y\in \{1,...,d\}$) or $\bm{e}_{z_{x}}+\bm{e}_{z_{y}}-\bm{e}_i$ ($x,y\in \{1,...,m\}$) to the left-hand side of (\ref{case22}), which must belong to $\{\bm{u}_{i_1}+\bm{u}_{j_1},...,\bm{u}_{i_t}+\bm{u}_{j_t}\}$. However, these elements cannot be written as the sum of two elements in $\{\bm{u}_1,...,\bm{u}_n\}$, contradicting (\ref{case22}). Thus, case 2 is impossible.\\
 
	\vskip 0.1cm
 \textbf{case 3:} Suppose $\bm{e}_{i^{'}_k}+\bm{e}_{j^{'}_k}-\bm{Z}=\bm{0}$, $k \in \{1,...,t\}$. For convenience, we write $\bm{e}_{i^{'}_k}+\bm{e}_{j^{'}_k}$ as $\bm{e}_i+\bm{e}_j$ and $i'_k=i$, $j'_k=j$.
 This implies that the node associated with $\bm{e}_i$ is adjacent to the node associated with $\bm{e}_j$	
in the graph $H'$. For notational simplicity, we hereafter refer to $\bm{e}_i$ and $\bm{e}_j$ as adjacent and continue to use this notation. Furthermore, (\ref{duiying2}) becomes
	\begin{equation}\label{case3}
		\begin{aligned}
			&\{-\bm{e}_i-\bm{e}_j,\bm{e}_{z_1}-\bm{e}_i-\bm{e}_j,...,\bm{e}_{z_{n-2}}-\bm{e}_i-\bm{e}_j,-\bm{e}_i,-\bm{e}_j,\bm{e}_{k_1}-\bm{e}_{i},...,\bm{e}_{k_s}-\bm{e}_{i},\bm{e}_{h_1}-\bm{e}_{j},...,\\ 
			&
			\bm{e}_{h_m}-\bm{e}_{j},\bm{e}_{p_{1}}+\bm{e}_{q_{1}}-\bm{e}_i-\bm{e}_j,...,\bm{e}_{p_{l}}+\bm{e}_{q_{l}}-\bm{e}_i-\bm{e}_j\}=\{\bm{u}_1,...,\bm{u}_n,\bm{u}_{i_1}+\bm{u}_{j_1},...,\bm{u}_{i_t}+\bm{u}_{j_t}\},  
		\end{aligned}
	\end{equation}
	where $s+m+l=t$, $\bm{e}_{z_1},...,\bm{e}_{z_{n-2}}$ are the $n-2$ elements in the set $\{\bm{e}_{1},...,\bm{e}_{n}\}$ except for $\bm{e}_i$ and $\bm{e}_j$; $\bm{e}_{k_1},...,\bm{e}_{k_s}$ are adjacent to \( \bm{e}_j \), and $\bm{e}_{h_1},..,\bm{e}_{h_m}$ are adjacent to \( \bm{e}_i \). Additionally, $\{\bm{e}_{k_1},...,\bm{e}_{k_s}\}\subseteq \{\bm{e}_{z_1},...,\bm{e}_{z_{n-2}}\}$ and $\{\bm{e}_{h_1},..,\bm{e}_{h_m}\} \subseteq\{\bm{e}_{z_1},...,\bm{e}_{z_{n-2}}\}$. Since graph $H'$ is connected, $\{\bm{e}_{k_1},...,\bm{e}_{k_s}\}$ and $\{\bm{e}_{h_1},..,\bm{e}_{h_m}\}$ cannot both be empty. Next, we will discuss separately based on the adjacency of the elements in set $\{\bm{e}_{z_1},...,\bm{e}_{z_{n-2}}\}$ to both $\bm{e}_i$ and $\bm{e}_j$.

	If each element in the set $\{\bm{e}_{z_1},...,\bm{e}_{z_{n-2}}\}$ is adjacent to either $\bm{e}_i$ or $\bm{e}_j$, it follows that $\{\bm{e}_{z_1},...,\bm{e}_{z_{n-2}}\}\subseteq\{\bm{e}_{k_1},...,\bm{e}_{k_s}, \bm{e}_{h_1},..,\bm{e}_{h_m}\}$. 
	
 \begin{itemize}
 \item Suppose $\{\bm{e}_{z_1},...,\bm{e}_{z_{n-2}}\}=\{\bm{e}_{h_1},..,\bm{e}_{h_m}\}$ and $\{\bm{e}_{k_1},...,\bm{e}_{k_s}\}=\emptyset$. This implies that every element in $\{\bm{e}_{z_1},...,\bm{e}_{z_{n-2}}\}$ is adjacent to $\bm{e}_i$ but not to $\bm{e}_j$. In this case, $\bm{e}_j$ is adjacent only to $\bm{e}_i$, and we have
 
 \begin{equation}\label{case31}
 	\begin{aligned}
 		&\{-\bm{e}_i-\bm{e}_j,\bm{e}_{z_1}-\bm{e}_i-\bm{e}_j,...,\bm{e}_{z_{n-2}}-\bm{e}_i-\bm{e}_j,-\bm{e}_i,-\bm{e}_j, \bm{e}_{z_1}-\bm{e}_{j},...,
 		\bm{e}_{z_{n-2}}-\bm{e}_{j},\\
 		&\bm{e}_{p_{1}}+\bm{e}_{q_{1}}-\bm{e}_i-\bm{e}_j,...,\bm{e}_{p_{l}}+\bm{e}_{q_{l}}-\bm{e}_i-\bm{e}_j\}=\{\bm{u}_1,...,\bm{u}_n,\bm{u}_{i_1}+\bm{u}_{j_1},...,\bm{u}_{i_t}+\bm{u}_{j_t}\}. 
 	\end{aligned}
 \end{equation}
 For the set on the left-hand side of (\ref{case31}), the elements $-\bm{e}_i,-\bm{e}_j,\bm{e}_{z_1}-\bm{e}_{j},...,\bm{e}_{z_{n-2}}-\bm{e}_{j}$ cannot be expressed as the sum of any two elements within this set. It follows that $\{\bm{u}_1,...,\bm{u}_n\}=\{-\bm{e}_i,-\bm{e}_j,\bm{e}_{z_1}-\bm{e}_{j},...,\bm{e}_{z_{n-2}}-\bm{e}_{j}\}$.  Given the presence of triangular cycles in graph $H'$, there will be the element $\bm{e}_{z_x}+\bm{e}_{z_y}-\bm{e}_i-\bm{e}_j$ (with $x,y\in \{1,...,n-2\}$) in the set on the left-hand side of (\ref{case31}), and it belongs to the set $\{\bm{u}_{i_1}+\bm{u}_{j_1},...,\bm{u}_{i_t}+\bm{u}_{j_t}\}$. However, this element cannot be decomposed into the sum of two elements in $\{\bm{u}_1,...,\bm{u}_n\}$, contradicting the equality in (\ref{case31}). Thus, the case where $\{\bm{e}_{z_1},...,\bm{e}_{z_{n-2}}\}=\{\bm{e}_{h_1},...,\bm{e}_{h_m}\}$ and $\{\bm{e}_{k_1},...,\bm{e}_{k_s}\}=\emptyset$ is impossible. By symmetry, the case where $\{\bm{e}_{z_1},...,\bm{e}_{z_{n-2}}\}=\{\bm{e}_{k_1},...,\bm{e}_{k_s}\}$ and $\{\bm{e}_{h_1},...,\bm{e}_{h_m}\}=\emptyset$ is also ruled out. 
 \item  Suppose $\{\bm{e}_{z_1},...,\bm{e}_{z_{n-2}}\}\subseteq\{\bm{e}_{k_1},...,\bm{e}_{k_s}, \bm{e}_{h_1},..,\bm{e}_{h_m}\}$ with $\{\bm{e}_{h_1},..,\bm{e}_{h_m}\} \neq \emptyset$ and $\{\bm{e}_{k_1},...,\bm{e}_{k_s}\}\neq\emptyset$. This scenario decomposes into two subcases:
 \begin{itemize}
 	\item [(i).] If $\{\bm{e}_{z_1},...,\bm{e}_{z_{n-2}}\}\subseteq\{\bm{e}_{k_1},...,\bm{e}_{k_s}, \bm{e}_{h_1},..,\bm{e}_{h_m}\}$ and $\{\bm{e}_{h_1},..,\bm{e}_{h_m}\} \cap \{\bm{e}_{k_1},...,\bm{e}_{k_s}\}=\emptyset$, it follows that elements of 
 	  $\{\bm{e}_{z_1},...,\bm{e}_{z_{n-2}}\}$ are partitioned into those adjacent to $\bm{e}_i$ and those adjacent to $\bm{e}_j$, with no overlap. This is impossible: graph $H'$ contains a node adjacent to all other nodes, contradicting the partition.
 	
 	\item [(ii).] If $\{\bm{e}_{z_1},...,\bm{e}_{z_{n-2}}\}\subseteq\{\bm{e}_{k_1},...,\bm{e}_{k_s}, \bm{e}_{h_1},..,\bm{e}_{h_m}\}$ and $\{\bm{e}_{h_1},..,\bm{e}_{h_m}\} \cap \{\bm{e}_{k_1},...,\bm{e}_{k_s}\}\neq \emptyset$, there exist elements in $\{\bm{e}_{z_1},...,\bm{e}_{z_{n-2}}\}$ adjacent to both $\bm{e}_i$ and $\bm{e}_j$. Consider equation \eqref{case3}: the left-hand side set includes elements $-\bm{e}_i,-\bm{e}_j,\bm{e}_{k_1}-\bm{e}_{i},...,\bm{e}_{k_s}-\bm{e}_{i},\bm{e}_{h_1}-\bm{e}_{j},...,
 	\bm{e}_{h_m}-\bm{e}_{j}$, none of which can be expressed as the sum of two elements in the set. The count of such elements exceeds $n$, contradicting the cardinality of $\{\bm{u}_1,...,\bm{u}_n\}$ in \eqref{case3}. Thus, this subcase is impossible.
 	\end{itemize}
 \end{itemize}
	 
If there exist elements in $\{\bm{e}_{z_1},...,\bm{e}_{z_{n-2}}\}$ that are neither adjacent to $\bm{e}_i$ nor to $\bm{e}_j$, denote the set of such elements by $\{\bm{e}_{z''_{1}},...,\bm{e}_{z''_{w}}\}$. Let  $\{\bm{e}_{z_1},...,\bm{e}_{z_{n-2}}\}\setminus \{\bm{e}_{z''_{1}},...,\bm{e}_{z''_{w}}\}=\{\bm{e}_{z^{'}_{1}},...,\bm{e}_{z^{'}_{d}}\}$, where each element of $\{\bm{e}_{z'_1},...,\bm{e}_{z'_{d}}\}$ is adjacent to either $\bm{e}_i$ or $\bm{e}_j$, and $\{\bm{e}_{z^{'}_{1}},...,\bm{e}_{z^{'}_{d}}\}\neq \emptyset$. It follows that $\{\bm{e}_{k_1},...,\bm{e}_{k_s}\}\subseteq \{\bm{e}_{z'_1},...,\bm{e}_{z'_{d}}\}$,\  $\{\bm{e}_{h_1},..,\bm{e}_{h_m}\} \subseteq\{\bm{e}_{z'_1},...,\bm{e}_{z'_{d}}\}$, and $\{\bm{e}_{z'_1},...,\bm{e}_{z'_{d}}\}\subseteq \{\bm{e}_{k_1},...,\bm{e}_{k_s},\bm{e}_{h_1},..,\bm{e}_{h_m}\}$.
Similarly, we discuss separately based on the adjacency of the elements in set $\{\bm{e}_{z'_1},...,\bm{e}_{z'_{d}}\}$ to both $\bm{e}_i$ and $\bm{e}_j$.

Since graph $H'$ contains a node that is adjacent to all other nodes, the following three cases are ruled out. 
\begin{itemize}
	\item [(i).] $\{\bm{e}_{z'_1},...,\bm{e}_{z'_{d}}\}=\{\bm{e}_{h_1},..,\bm{e}_{h_m}\}$ and $\{\bm{e}_{k_1},...,\bm{e}_{k_s}\}=\emptyset$.
	\item [(ii).] $\{\bm{e}_{z'_1},...,\bm{e}_{z'_{d}}\}=\{\bm{e}_{k_1},...,\bm{e}_{k_s}\}$ and $\{\bm{e}_{h_1},..,\bm{e}_{h_m}\}=\emptyset$.
	\item [(iii).] $\{\bm{e}_{h_1},..,\bm{e}_{h_m}\} \neq \emptyset$, $\{\bm{e}_{k_1},...,\bm{e}_{k_s}\}\neq\emptyset$, and $\{\bm{e}_{h_1},..,\bm{e}_{h_m}\} \cap \{\bm{e}_{k_1},...,\bm{e}_{k_s}\}=\emptyset$.
\end{itemize} 

Now, we only consider the following case.

	If $\{\bm{e}_{z'_1},...,\bm{e}_{z'_{d}}\}\subseteq\{\bm{e}_{k_1},...,\bm{e}_{k_s}, \bm{e}_{h_1},..,\bm{e}_{h_m}\}$ and $\{\bm{e}_{h_1},..,\bm{e}_{h_m}\} \cap \{\bm{e}_{k_1},...,\bm{e}_{k_s}\}\neq \emptyset$, there exist elements in  $\{\bm{e}_{z'_1},...,\bm{e}_{z'_{d}}\}$ adjacent to both $\bm{e}_i$ and $\bm{e}_j$. Consider the equation
	\begin{equation}\label{case32}
		\begin{aligned}
			&\{-\bm{e}_i-\bm{e}_j,\bm{e}_{z'_1}-\bm{e}_i-\bm{e}_j,...,\bm{e}_{z'_{d}}-\bm{e}_i-\bm{e}_j,\bm{e}_{z''_1}-\bm{e}_i-\bm{e}_j,...,\bm{e}_{z''_{w}}-\bm{e}_i-\bm{e}_j,-\bm{e}_i,-\bm{e}_j,\\
			&\bm{e}_{k_1}-\bm{e}_{i},...,\bm{e}_{k_s}-\bm{e}_{i},\bm{e}_{h_1}-\bm{e}_{j},...,
			\bm{e}_{h_m}-\bm{e}_{j},\bm{e}_{p_{1}}+\bm{e}_{q_{1}}-\bm{e}_i-\bm{e}_j,...,\bm{e}_{p_{l}}+\bm{e}_{q_{l}}-\bm{e}_i-\bm{e}_j\}\\
			&=\{\bm{u}_1,...,\bm{u}_n,\bm{u}_{i_1}+\bm{u}_{j_1},...,\bm{u}_{i_t}+\bm{u}_{j_t}\},  
		\end{aligned}
	\end{equation}	
In the left-hand side of (\ref{case32}), the elements $-\bm{e}_i,-\bm{e}_j,\bm{e}_{k_1}-\bm{e}_{i},...,\bm{e}_{k_s}-\bm{e}_{i},\bm{e}_{h_1}-\bm{e}_{j},...,
		\bm{e}_{h_m}-\bm{e}_{j}, \bm{e}_{z''_1}-\bm{e}_i-\bm{e}_j,...,\bm{e}_{z''_{w}}-\bm{e}_i-\bm{e}_j$ cannot be expressed as the sum of any two elements in the set. Since the cardinality of these elements exceeds $n$, the equality in \eqref{case32} fails, ruling out this case.

	In summary, if the polytopes $P(H)$ and $P(H')$ are unimodularly isomorphic, then $\bm{Z}=\bm{0}$ and $\{\bm{u}_1,...,\bm{u}_n\}=\{\bm{e}_1,...,\bm{e}_n\}$, $\{\bm{u}_{i_1}+\bm{u}_{j_1},...,\bm{u}_{i_t}+\bm{u}_{j_t}\}=\{\bm{e}_{i^{'}_{1}}+\bm{e}_{j^{'}_{1}},...,\bm{e}_{i^{'}_{t}}+\bm{e}_{j^{'}_{t}}\}$. This implies
	\begin{gather*}
		\{\bm{u}_1,...,\bm{u}_n\}=\{U\bm{e}_1,...,U\bm{e}_n\}=\{\bm{e}_1,...,\bm{e}_n\},\\
		\{\bm{u}_{i_1}+\bm{u}_{j_1},...,\bm{u}_{i_t}+\bm{u}_{j_t}\}=\{U(\bm{e}_{i_1}+\bm{e}_{j_1}),...,U(\bm{e}_{i_t}+\bm{e}_{j_t})\}=\{\bm{e}_{i^{'}_{1}}+\bm{e}_{j^{'}_{1}},...,\bm{e}_{i^{'}_{t}}+\bm{e}_{j^{'}_{t}}\},
	\end{gather*}
	indicating that there exists a permutation of the (node) set $\{\bm{e}_1,...,\bm{e}_n\}$ inducing a bijection between$\{\bm{e}_{i_1}+\bm{e}_{j_1},...,\bm{e}_{i_t}+\bm{e}_{j_t}\}$ and $\{\bm{e}_{i^{'}_{1}}+\bm{e}_{j^{'}_{1}},...,\bm{e}_{i^{'}_{t}}+\bm{e}_{j^{'}_{t}}\}$. Specifically, $\bm{e}_{i_k}+\bm{e}_{j_k}$ is an edge of $H$ if and only if $U(\bm{e}_{i_k}+\bm{e}_{j_k})$ is an edge of $H'$. Hence, graphs  $H$ and $H'$ are isomorphic, and so are $G$ and $G'$.
\end{proof}
The above theorem indicates that the unimodular isomorphism problem of convex lattice polytopes is graph isomorphism hard.

\section{Statistical Zero-Knowledge Proof of Knowledge}
In this section, analogously to the protocol for lattice isomorphism \cite{Ducas}, we present a statistical zero-knowledge proof system for unimodular isomorphism of lattice polytopes. To do
so, we first define our average-case distribution within a unimodular isomorphism class $[P]$ and show how to sample from it.
\subsection{An Average-Case Distribution}
First, for any $n$-dimensional convex lattice polytope $P\subset \mathbb{Z}^{n}$, we define the average-case distribution over the unimodular isomorphism class $[P]$. This work uses Gaussian sampling \cite{Gentry} and techniques analogous to those in \cite{Ducas}. As we will see in the subsequent proof, the Gaussian sampling ensures that the output distribution depends solely  on the isomorphism class $[P]$, and not on the specific input lattice polytope. Without loss of generality, we can use $V_{P}^{ord}$ to represent $P$, where $V_{P}^{ord}$ is the vertex matrix obtained by arranging the column vectors of $V_P$ in lexicographic ascending order. Each $n$-dimensional lattice polytope $P\subset \mathbb{Z}^{n}$ has a unique vertex matrix $V^{ord}_{P}$, such that distinct matrices $V^{ord}_{P}$ and $V^{ord}_{P'}$ correspond to distinct convex lattice polytopes $P$ and $P'$.

We begin with the linear algebra procedure. Given the vertex matrix $V^{ord}_{P}\in \mathbb{Z}^{n\times d}$ of an $n$-dimensional convex lattice polytope $P\subset \mathbb{Z}^{n}$, a set of linearly independent vectors $Y:=(\bm{y}_1,...,\bm{y}_n)^{t}\in \mathbb{Z}^{n\times n}$, and an integer vector 
 $\bm{z}\in \mathbb{Z}^{n}$, this procedure returns a convex lattice polytope $R\in [P]$. Here, $\bm{y}_1,...,\bm{y}_n$ and $\bm{z}$ are sampled from the Gaussian distribution $D_{Q,s}$, where $Q=(V_{P}^{ord}-b_P)(V_{P}^{ord}-b_P)^{t}$ and $b_{P}$ denotes the vertex average of $P$. The detailed steps of this process are provided below.

Firstly, we use the polynomial time algorithm in Lemma \ref{polytime} to obtain a unimodular matrix and a convex polytope $P_1$.
\begin{lem}\label{polytime}
	There is a polynomial time
	algorithm $(V_{P_1},U)\leftarrow \bm{\mathrm{PolyTime}}(V^{ord}_{P},Y)$ that on input a vertex matrix $V^{ord}_{P}=(\bm{v}_1,...,\bm{v}_d)$, and a set of linearly independent vectors $Y:=(\bm{y}_1,...,\bm{y}_n)^{t}\in \mathbb{Z}^{n\times n}$, outputs a transformation $U\in GL_{n}(\mathbb{Z})$ and a vertex matrix $V_{P_1}=U(V^{ord}_{P}-b_{P})$ of the $n$-dimensional convex polytope $P_1=U(P-b_P)$.
\end{lem}
\begin{proof}
	For the input vertex matrix $V^{ord}_{P}$, it is easy to obtain $b_P$ and $(\bm{v}_1-b_{P},...,\bm{v}_d-b_{P})$. For the integer matrix $Y=(\bm{y}_1,...,\bm{y}_n)^{t}$,  there exists the unique transformation $U\in GL_{n}(\mathbb{Z})$ such that $T=YU^{-1}$ is the canonical lower triangular Hermite Normal Form of $Y$. In addition, the Hermite Normal Form of an integer matrix and the corresponding unimodular transformation can be computed in polynomial time \cite{Cohen}.
\end{proof}
Secondly, by arranging the column vectors of the vertex matrix $V_{P_1}$ in lexicographical ascending order, we can obtain $V_{P_{1}}^{ord}=(\bm{v}_{1}^{ord},...,\bm{v}_{d}^{ord})=U(\bm{v}_1-b_{P},...,\bm{v}_d-b_{P})M$, for some permutation matrix $M$. We denote this lexicographical ordering process as \textbf{LexiOrder()}.

Thirdly, by translating $P_1$ by $\bm{v}_{1}^{ord}$, we can obtain the vertex matrix $V_{P_2}^{ord}$ of the $n$-dimensional convex polytope $P_2$, where $V_{P_2}^{ord}=(\bm{0},...,\bm{v}_{d}^{ord}-\bm{v}_{1}^{ord})$, and $\bm{v}_{1}^{ord}=U\bm{v}_{i}-Ub_{P}$ for some $i \in \{1, 2, ..., d\}$.

Finally, we can obtain the vertex matrix of the $n$-dimensional convex polytope R as follows.
\begin{eqnarray}\label{R}
	\begin{aligned}
		V^{ord}_{R}&= V_{P_2}^{ord}+(U^{-1})^{t}\bm{z}\\
		&=(\bm{0},...,\bm{v}_{d}^{ord}-\bm{v}_{1}^{ord})+((U^{-1})^{t}\bm{z},...,(U^{-1})^{t}\bm{z})\\
		&=(\bm{v}_{1}^{ord},...,\bm{v}_{d}^{ord})+((U^{-1})^{t}\bm{z}-\bm{v}_{1}^{ord},...,(U^{-1})^{t}\bm{z}-\bm{v}_{1}^{ord})\\
		&=U(\bm{v}_1,...,\bm{v}_d)M+((U^{-1})^{t}\bm{z}-\bm{v}_{1}^{ord}-Ub_{P},...,(U^{-1})^{t}\bm{z}-\bm{v}_{1}^{ord}-Ub_{P})\\
		&=U(\bm{v}_1,...,\bm{v}_d)M+((U^{-1})^{t}\bm{z}-U\bm{v}_{i},...,(U^{-1})^{t}\bm{z}-U\bm{v}_{i}). 
	\end{aligned}   
\end{eqnarray}
Since the vector $\bm{Z}=(U^{-1})^{t}\bm{z}-U\bm{v}_{i}\in \mathbb{Z}^{n}$, then $R$
is an $n$-dimensional convex lattice polytope and $R=UP+\bm{Z}$.

We summarize the above process as the following Algorithm 1.
\begin{algorithm}[!h]\label{extract}
	\caption{$(V^{ord}_{R},U,\bm{Z})\leftarrow \bm{\mathrm{Extract}}(V^{ord}_{P},Y,\bm{z})$.}
	
	\KwData{The vertex matrix $V^{ord}_{P}\in \mathbb{Z}^{n\times d}$ of an $n$-dimensional convex lattice polytope $P\subset \mathbb{Z}^{n}$, a set of linearly independent vectors
		$Y=(\bm{y}_1,...,\bm{y}_n)^{t}\in \mathbb{Z}^{n\times n}$, and  an integer vector $\bm{z}\in \mathbb{Z}^{n}$.}
	\KwResult{ The vertex matrix $V^{ord}_{R}\in \mathbb{Z}^{n\times d}$ of $R=UP+\bm{Z}$, with a transformation $U\in GL_{n}(\mathbb{Z})$
		and an integer vector $\bm{Z}\in \mathbb{Z}^{n}$.}
	$(V_{P_1},U)\leftarrow \bm{\mathrm{PolyTime}}(V^{ord}_{P},Y)$;\\
	$V_{P_{1}}^{ord}\leftarrow \bm{\mathrm{LexiOrder}}(V_{P_1})$;\\
	$V_{P_2}^{ord}\leftarrow V_{P_{1}}^{ord}-\bm{v}_{1}^{ord}$;\\
	$\bm{\mathrm{Return}}:V_{R}^{ord}=V_{P_2}^{ord}+(U^{-1})^{t}\bm{z}$, $U$ and the integer vector
	$\bm{Z}=(U^{-1})^{t}\bm{z}-\bm{v}_{1}^{ord}-Ub_{P}$;
\end{algorithm}

In the second step, the lexicographical ordering of the column vectors of \( V_{P_1} \in \mathbb{R}^{n \times d} \) can be achieved in \( O(nd \log d) \) time by comparing entries row-wise, which is polynomial in the input size (since \( d \) denotes the number of vertices and \( n \) the dimension). Combined with the polynomial-time guarantees of \( \bm{\mathrm{PolyTime}} \) and other linear operations, we derive the following lemma.

\begin{lem}\label{polytime1}
	The running time of Algorithm 1 is polynomial with respect to the parameters \( n \) (dimension) and \( d \) (number of vertices).
\end{lem}

Now we can formally define our average-case distribution for a Gaussian parameter $s>0$ and show the
distribution only depends on the isomorphism class of the input lattice polytope and not on the specific representative.
\begin{defn}\label{fenbu}
	Given an $n$-dimensional convex lattice polytope $P\subset \mathbb{Z}^{n}$, we define the Gaussian form distribution $D_{s}([P])$ over $[P]$ with parameter $s>0$ algorithmically as follows:
	\begin{itemize}
		\item [1.] Fix a representative $P\in [P]$. 
		\item [2.] Give the quadratic form $Q=(V_{P}^{ord}-b_P)(V_{P}^{ord}-b_P)^{t}\in S^{>0}_{n}$.
		\item [3.] Sample $n$ vectors $(\bm{y}_1,...,\bm{y}_n)^{t}=:Y$ from $D_{Q,s}$.
		Repeat until linearly independent.
		\item [4.] Sample an integer vector $\bm{z}$ from $D_{Q,s}$.
		\item [5.] $(V_{R}^{ord},U,\bm{Z})$ $\leftarrow$ $\bm{\mathrm{Extract}}(V_{P}^{ord},Y,\bm{z})$.
		\item [6.]  Return $V_{R}^{ord}$.
	\end{itemize}
\end{defn}
\begin{proof}\label{wuguan}
	We have to show that the distribution is well-defined over the unimodular isomorphism
	class $[P]$, i.e., for any input representative $P'\in [P]$, the output distribution
	should be identical. 
	For $P$ and  $P'=VP+\bm{v}\in [P]$, we derive the relation for the $n\times n$ positive definite real symmetric matrices $Q$ and $Q'$:
		\begin{equation*}
			Q'=(V_{P'}^{ord}-b_{P'})(V_{P'}^{ord}-b_{P'})^{t}=V(V_{P}^{ord}-b_P)(V_{P}^{ord}-b_P)^{t}V^{t}=VQV^{t}.
		\end{equation*}
According to (\ref{Dps}), for  $\bm{x},\bm{x}'\in \mathbb{Z}^{n}$, $\mathop{\Pr} \limits_{X\thicksim D_{Q,s}} [X=\bm{x}] =\mathop{\Pr} \limits_{X\thicksim D_{Q',s}} [X=\bm{x}'] $ if and only if  $\frac{\rho_{Q,s}(\bm{x})}{\rho_{Q,s}(\mathbb{Z}^{n})}=\frac{\rho_{Q',s}(\bm{x}')}{\rho_{Q',s}(\mathbb{Z}^{n})}$. Furthermore, since $V\in GL_n(\mathbb{Z})$, then
\begin{align*}
	\rho_{Q,s}(\mathbb{Z}^{n})=\sum_{\bm{x}\in \mathbb{Z}^{n}}\mathrm{exp}\left(-\pi\frac{\bm{x}^{t}Q\bm{x}}{s^{2}}\right)=\sum_{\bm{x}'\in \mathbb{Z}^{n}}\mathrm{exp}\left(-\pi\frac{(\bm{x}')^{t}Q'\bm{x}'}{s^{2}}\right)=\rho_{Q',s}(\mathbb{Z}^{n}).
\end{align*}
 Thus, we have
	\begin{equation*}
	\mathop{\Pr} \limits_{X\thicksim D_{Q,s}} [X=\bm{x}] =\mathop{\Pr} \limits_{X\thicksim D_{Q',s}} [X=\bm{x}']\Longleftrightarrow 	\rho_{Q,s}(\bm{x})=\rho_{Q',s}(\bm{x}') \Longleftrightarrow \bm{x}^{t}Q\bm{x}=(\bm{x}')^{t}Q'\bm{x}' \Longleftrightarrow \bm{x}'=(V^{t})^{-1}\bm{x}.
	\end{equation*}
Therefore, if step 3 and step 4 on input $Q$ finish with $Y$ and $\bm{z}$,
	then they on input $Q'$ finish with $Y'=YV^{-1}$ and $\bm{z}'=(V^{t})^{-1}\bm{z}$ with the same probability.
	
	Let $\bm{V}$ be a random vertex matrix of convex lattice polytopes. Let $\bm{U}$ be a random unimodular matrix and $\bm{X}$ be a random integer vector. We consider the outputs $(V^{ord}_{R},U,\bm{Z})\leftarrow \bm{\mathrm{Extract}}(V^{ord}_{P},Y,\bm{z})$ and $(V^{ord}_{R'},U',\bm{Z}')\leftarrow \bm{\mathrm{Extract}}(V^{ord}_{P'},Y',\bm{z}')$. 
	
	By Lemma \ref{polytime}, it follows the outputs $(V_{P_1},U)\leftarrow \bm{\mathrm{PolyTime}}(V^{ord}_{P},Y)$, $(V_{P^{'}_{1}},U')\leftarrow \bm{\mathrm{PolyTime}}(V^{ord}_{P'},Y')$. From the canonicity of the Hermite Normal Form, we immediately obtain
	 that $Y'(U')^{-1}=YV^{-1}(U')^{-1}=H=YU^{-1}$, and thus $U'=UV^{-1}$ and $\Pr[\bm{U}=U]=\Pr[\bm{U}=U']$. Furthermore, we have 
	 	\begin{align*}
	 	P_{1}^{'}=U'(P'-b_{P'})=UV^{-1}[V(P-b_{P})]=U(P-b_{P})=P_1.		
	 \end{align*}
	Thus, $V_{P_{1}}^{ord}=V_{P'_{1}}^{ord}$, $\bm{v}_{1}^{ord}=(\bm{v}')_{1}^{ord}$, and $V_{P_2}^{ord}=V_{P_{2}^{'}}^{ord}$. Finally, we can obtain 
	\begin{gather*}
	V_{R'}^{ord}=V_{P_{2}^{'}}^{ord}+((U')^{-1})^{t}\bm{z}',\  V_{R}^{ord}=V_{P_2}^{ord}+(U^{-1})^{t}\bm{z};\\
	\bm{Z}=(U^{-1})^{t}\bm{z}-\bm{v}_{1}^{ord}-Ub_{P},\ \bm{Z}'=(U'^{-1})^{t}\bm{z}'-(\bm{v}')_{1}^{ord}-U'b_{P'}.
	\end{gather*}
Since $\Pr[\bm{U}=U]=\Pr[\bm{U}=U']$ and $\Pr[X=\bm{z}]=\Pr[X=\bm{z}']$, then $\Pr[\bm{X}=\bm{Z}]=\Pr[\bm{X}=\bm{Z}']$ and $\Pr[\bm{V}=V_{R}^{ord}]=\Pr[\bm{V}=V_{R'}^{ord}]$. Moreover, we have 
\begin{align*}
	V_{R'}^{ord}=V_{P_{2}^{'}}^{ord}+(U'^{-1})^{t}\bm{z}'=V_{P_2}^{ord}+(V^{-1}VU^{-1})^{t}\bm{z}=V_{P_2}^{ord}+(U^{-1})^{t}\bm{z}=V_{R}^{ord}.
\end{align*}

	Then, in step 5, if $\bm{\mathrm{Extract}}(V_{P}^{ord},Y, \bm{z})$ returns $(V_{R}^{ord},U,\bm{Z})$, then $\bm{\mathrm{Extract}}(V_{P'}^{ord},Y', \bm{z}')$ returns $(V_{R}^{ord},U',\bm{Z'})$ with the same probability.
So the distribution is well-defined over $[P]$.  
\end{proof}

Given the algorithmic definition of $D_{s}([P])$, an efficient sampling algorithm can be derived with minor adaptations. Firstly, we need to efficiently sample
from $D_{Q,s}$, which puts some constraints on the parameter $s$ depending on the
reducedness of the quadratic form $Q$. Secondly, we adopt the method in \cite{Ducas} to obtain $n$ linearly independent vectors, that is, we sample vectors one by one and only add them to $Y$ if they are independent. Then,  we still
need the additional constraint $s \geqslant \lambda_{n}(Q)$  to show that this succeeds with a
polynomial amount of samples. 
\begin{algorithm}[!h]
	\caption{Sampling from $D_{s}([P])$.}
	
	\KwData{An $n$-dimensional convex lattice polytope $P\subset \mathbb{Z}^{n}$, the quadratic form $Q=(V_{P}^{ord}-b_P)(V_{P}^{ord}-b_P)^{t}$, and a parameter $s\geqslant max\{ \lambda_{n}(Q), \parallel B^{*}_{Q}\parallel \cdot \sqrt{ln(2n+4)/\pi}\}$.}
	\KwResult{Sample $R=UP+\bm{Z}$ from $D_{s}([P])$, with a transformation $U\in GL_{n}(\mathbb{Z})$
		and an integer vector $\bm{Z}\in \mathbb{Z}^{n}$.}
	$Y\leftarrow \varnothing$;\\
	
	\While{$|Y|<n$}{Sample $\bm{x}\leftarrow D_{Q,s}$;\\
		\If{$\bm{x}^{t}\notin \mathrm{span}(Y)$}{Append $\bm{x}^{t}$ to $Y$;}
	}
	Sample $\bm{z}\leftarrow D_{Q,s}$;\\
	$(V_{R}^{ord},U, \bm{Z})\leftarrow \bm{\mathrm{Extract}}(V_{P}^{ord},Y,\bm{z})$;\\
\end{algorithm}
\begin{thm}\label{polytime2}
	For any $n$-dimensional convex lattice polytope $P\subset \mathbb{Z}^{n}$, let $Q=(V_{P}^{ord}-b_P)(V_{P}^{ord}-b_P)^{t}$ and the parameter
	\begin{align*}
		s\geqslant \max\{ \lambda_{n}(Q), \parallel B^{*}_{Q}\parallel \cdot \sqrt{ln(2n+4)/\pi}\}.
	\end{align*}
	Algorithm 2 runs in expected polynomial time and returns $(R,U, \bm{Z})\in [P]\times GL_n(\mathbb{Z})\times \mathbb{Z}^{n}$, where $R$ is a sample from $D_{s}([P])$, $U$ is a unimodular matrix and $\bm{Z}$ is an integer vector such that $R=UP+\bm{Z}$. Furthermore,  the isomorphism $\mathcal{U}=(U,\bm{Z})$ is uniform over the set of isomorphisms from $P$ to $R$. 
\end{thm}
\begin{proof}
	By Lemmas \ref{caiyang} and \ref{polytime1}, every step in Algorithm 2 runs in polynomial
	time. It remains to prove that the number of iterations is polynomially bounded.  Since we adopt the method in \cite{Ducas} to sample $n$ linearly independent vectors, we reproduce the proof here for the sake of completeness. Let the random variable $K$ denote the number of samples before we find $n$ independent ones. If $|Y|<n$, then because $s \geqslant \lambda_{n}(Q)$, by the Lemma 5.1 in \cite{Regev}, we have that every newly sampled vector $\bm{x} \leftarrow D_{Q,s}$ is not in the span of $Y$ with constant probability at least $p:=1-(1+e^{-\pi})^{-1}>0$. So $K$ is bounded from
	above by a negative binomial distribution for $n$ successes with success probability
	$p$, which implies that $E[K]\leqslant \frac{n}{p}$. In particular, according to Corollary 5.1 in \cite{Regev}, we know $\mathrm{Pr}[K> n^{2}] \leqslant 2^{-\Omega (n)}$. When the while loop succeeds, the set $Y$ is distributed as $n$ vectors sampled from $D_{Q,s}$ under the linear independence condition, following exactly Definition \ref{fenbu}.
	
	 The set of isomorphisms from $P$ to $R$ is $\{\mathcal{UA}=(UA,U\bm{a}+\bm{Z})\in GL_n(\mathbb{Z})\times \mathbb{Z}^{n}:\mathcal{A}\in \mathrm{Aut}(P)\}$.  We now prove that the isomorphism $\mathcal{U}=(U,\bm{Z})$ is uniform over this set.
	Suppose that the algorithm terminates with a full-rank matrix $Y$ and an integer vector $\bm{z}$, returning $(V_{R}^{ord},U, \bm{Z})\leftarrow \bm{\mathrm{Extract}}(V_{P}^{ord},Y,\bm{z})$. By Lemma \ref{polytime} and Algorithm 1, we have $YU^{-1}=H$ and $\bm{Z}=(U^{-1})^{t}\bm{z}-\bm{v}_{1}^{ord}-Ub_{P}$. For any unimodular automorphism $\mathcal{A}\in \mathrm{Aut}(P)$ such that $P=AP+\bm{a}$, it follows that $P=A^{-1}P-A^{-1}\bm{a}$. Let $A^{-1}P-A^{-1}\bm{a}=:P'$. By the proof of Definition \ref{fenbu}, for $P'\in [P]$, the algorithm terminates with the full-rank matrix $Y'=YA$ and the integer vector $\bm{z}'=A^{t}\bm{z}$ with the same probability, and returns $(V_{R}^{ord},U', \bm{Z}')\leftarrow \bm{\mathrm{Extract}}(V_{P'}^{ord},Y', \bm{z}')$ with the same probability. Then, we have $\Pr[(\bm{U}, \bm{X})=(U,\bm{Z})]=\Pr[(\bm{U}, \bm{X})=(U',\bm{Z}')]$.
	
	 From Lemma \ref{polytime}, for $P'$, we have $Y'(U')^{-1}=YA(U')^{-1}=H=YU^{-1}$, so $U'=UA$. By Algorithm 1, we have $\bm{Z}'=(U'^{-1})^{t}\bm{z}'-(\bm{v}')_{1}^{ord}-U'b_{P'}$, where $(\bm{v}')_{1}^{ord}=\bm{v}_{1}^{ord}$. 
	  Then, we can obtain
	\begin{eqnarray}\label{R}
		\begin{aligned}
			\bm{Z}'&=(U'^{-1})^{t}\bm{z}'-(\bm{v}')_{1}^{ord}-U'b_{P'}\\
			&=((UA)^{-1})^{t}A^{t}\bm{z}-\bm{v}^{ord}_1-UA(A^{-1}b_{P}-A^{-1}\bm{a})\\
			&=(U^{-1})^{t}\bm{z}-\bm{v}^{ord}_1-Ub_{P}+U\bm{a}.
		\end{aligned}
	\end{eqnarray}
	Since $\bm{Z}=(U^{-1})^{t}\bm{z}-\bm{v}_{1}^{ord}-Ub_{P}$, then $\bm{Z}'=U\bm{a}+\bm{Z}$.
	
	 Thus, $\Pr[(\bm{U}, \bm{X})=(U,\bm{Z})]=\Pr[(\bm{U}, \bm{X})=(UA,U\bm{a}+\bm{Z})]$, which makes the returned transformation uniform over the set of isomorphisms $\{\mathcal{UA}=(UA,U\bm{a}+\bm{Z})\in GL_n(\mathbb{Z})\times \mathbb{Z}^{n}:\mathcal{A}\in \mathrm{Aut}(P)\}$ from $P$ to $R$. Specifically, the returned unimodular matrix $U$ is uniform on the set of unimodular transformations $\{UA: A\in \mathrm{Aut_u}(P)\}$ from $P$ to $R$, and the returned lattice translation $\bm{Z}$ is uniform on the set of lattice translations $\{U\bm{a}+\bm{Z}: \bm{a}\in \mathrm{Aut_t}(P)\}$ from $P$ to $R$.
\end{proof}
\subsection{Zero Knowledge Proof of Knowledge}
\setlength{\FrameRule}{1pt}
\setlength{\FrameSep}{10pt}
\setlength{\leftmargin}{0pt}
\begin{framed}

	\centerline{\textbf{Zero Knowledge Proof of Knowledge $\Sigma$}}
	
	\noindent Consider two unimodularly isomorphic public $n$-dimensional lattice polytopes $P_0, P_1 \subset \mathbb{Z}^{n}$ with $d$ vertices, a secret unimodular transformation $U \in GL_n(\mathbb{Z})$ such that $P_1-b_{P_1}=U(P_0-b_{P_0})$. Then, $Q_0=(V_{P_0}^{ord}-b_{P_0})(V_{P_0}^{ord}-b_{P_0})^{t}$ and $Q_1=(V_{P_1}^{ord}-b_{P_1})(V_{P_1}^{ord}-b_{P_1})^{t}=UQ_0U^{t}$. Given the public parameter
	\begin{align*}
		s \geqslant \max \left\{\lambda_n([Q_0]), \max \left\{\|B_{Q_0}^\ast\|, \|B_{Q_{1}}^\ast\|\right\} \cdot \sqrt{\ln(2n + 4) / \pi}\right\},
	\end{align*}

	\noindent we define the following protocol $\Sigma$ that gives a zero-knowledge proof of knowledge of an isomorphism between $P_0$ and $P_1$:\\ 
	
	$\begin{array}{ccc}
		\text{Prover} &   & \text{Verifier} \\
		\text{Sample $P' \gets \mathcal{D}_s([P_0])$ by Alg. 2,}\\
		\text{ together with $V$ and $\bm{Z}'$ }\\
		s.t.\ P' = VP_0+\bm{Z}' &   &\\
		& \xrightarrow{\quad P' \quad}  & \text{Sample $c \rightarrow \mathcal{U}(\{0,1\})$}\\
		\text{Compute $W=VU^{-c}$} & \xleftarrow{\quad c \quad} &\\
		& \xrightarrow{\quad W \quad} & \quad \text{$\quad$ Check if  $W\in GL_n(\mathbb{Z})$,}\quad\\  
		&     & \text{ and $\mathcal{V}(P'-b_{P'})=\mathcal{V}(WP_c-Wb_{P_c})$}.
		
	\end{array}$
\end{framed}
 \noindent \emph{Efficiency and completeness.} For the efficiency of $\Sigma$ we need to verify that Algorithm 2 runs in polynomial time, and indeed by Lemma \ref{polytime2} this is the case
because
\begin{align*}
	s\geq max\{ \lambda_{n}([Q_0]), \parallel B^{*}_{Q_0}\parallel \cdot \sqrt{ln(2n+4)/\pi}\}.
\end{align*}
For the $W$, we have that $W \in GL_n(\mathbb{Z})$ if and only if $W$ is integral and
$det(W)=\pm 1$, both of which are easy to check in polynomial time. To verify vertex set equality, we sort the elements of $\mathcal{V}(P'-b_{P'})$ and $\mathcal{V}(WP_c-Wb_{P_c})$  lexicographically, then compare them sequentially. If all corresponding elements match, the sets are equal; otherwise, they are not. This approach also runs in polynomial time.

For the completeness of $\Sigma$, note that when the prover executes the protocol
honestly we have $W :=V\cdot U^{-c}\in GL_n(\mathbb{Z})$ because $U$ and $V$ are both unimodular
by definition. Additionally, we have
\begin{align}\label{dmtxt}
	P'-b_{P'}=V(P_0-b_{P_0})=VU^{-c}(U^{c}(P_0-b_{P_0}))=W(P_c-b_{P_c}),
\end{align}
and thus the verifier accepts.
\vskip 0.2cm
\noindent \emph{Special soundness.} Suppose we have two accepting conversations $(P',0,W_{0})$ and $(P', 1, W_1)$ of $\Sigma$ where the first message is identical. The acceptance implies that $W_0, W_1 \in  GL_n(\mathbb{Z})$ and $W_{0}(P_0-b_{P_0})=P'-b_{P'}=W_1(P_1-b_{P_1})$, and thus $U':=W_{1}^{-1}W_0 \in GL_n(\mathbb{Z})$ gives an isomorphism from $P_0$ to $P_1$ as
\begin{align*}
	U'(P_0-b_{P_0})=W_{1}^{-1}(W_0(P_0-b_{P_0}))=W_{1}^{-1}(W_1(P_1-b_{P_1}))=P_1-b_{P_1}.
\end{align*}
We conclude that $\Sigma$ has the special soundness property.

\vskip 0.2cm
\noindent \emph{Special honest-verifier zero-knowledge.} 
Given any challenge $c$, we create a simulator that, given the
public inputs $P_0, P_1$, outputs an accepting conversation with the same probability distribution as that between an honest prover and verifier. We analyze the distribution of the three messages in the real conversation. The first message $P'$ is always distributed as $D_s([P_0])$, the challenge $c$ as $\mathcal{U}(\{0,1\})$, and $V$ is uniform over the set of unimodular transformations from $P_0$ to $P'$ by Theorem \ref{polytime2}. Given the challenge $c$, the honest prover computes $W=VU^{-c}$, where $U$ is a unimodular transformation from $P_0$ to $P_1$. The verifier obtains $P'=W(P_c-b_{P_c})+b_{P'}$; thus, by Theorem \ref{polytime2}, $W=VU^{-c}$ is uniform over the set of unimodular transformations from $P_c$ to $P'$.

To simulate this, we first sample the uniformly random challenge $c \leftarrow \mathcal{U}(\{0,1\})$. If $c=0$ we can proceed the same as in $\Sigma$ itself. Specifically, we sample $P' \leftarrow D_s([P_0])$ using Algorithm 2, together with  $V$ and $\bm{Z}'$ such that $P' = VP_0+\bm{Z}'$, and set $W:=V$. The final conversation $(P',0,W)$ is accepting by construction and follows the same distribution as during an honest execution conditioned on challenge $c=0$.

If $c=1$ we use the fact that $[P_0] = [P_1]$ and $D_s([P_0])=D_s([P_1])$. We use Algorithm 2
with representative $P_1$ as input instead of $P_0$. So again we obtain 
$P'\leftarrow D_s([P_1])=D_s([P_0])$, but now together with $W$ and $\bm{Z}'_1$ such that $P'=WP_1+\bm{Z}'_1$. The conversation $(P', 1, W)$ is accepting by construction, and $P'$ follows the same distribution $D_{s}([P_0])$. Additionally, by Theorem \ref{polytime2}, the unimodular transformation $W$ is indeed uniform over the set of unimodular transformations from $P_1$ to $P'$.

We conclude that $\Sigma$ has the special honest-verifier zero-knowledge property.

\section{The Algorithm for Unimodular Isomorphism Problem}
In this section, given $n$-dimensional convex lattice polytopes $P,P'\subset \mathbb{Z}^{n}$ with $d$ vertices, we describe a method to compute all unimodular affine transformations between $P$ and $P'$, in particular,  to decide the UIP. Specifically, we search for such transformations via label-preserving isomorphic minimum spanning trees of labeled graphs $\mathcal{GW}(P)$ and $\mathcal{GW}(P')$. Next, we present a detailed procedure as follows.

\vskip 0.2cm
\noindent \emph{The algorithm for UIP}. Firstly, given the vertex matrix $V_{P}$ of $P$, we find all edges of the polytope to construct a graph $G(P)$ defined by its vertices and one-dimensional edges. $G(P)$ can be constructed using the method proposed in \cite{Yang}.
\begin{thm}[\cite{Yang}]\label{veg}
	There exists a polynomial time algorithm \textbf{Graph}$(V_{P},n,d)$ that given the $n\times d$ vertex matrix $V_{P}$ of a lattice polytope $P$, returns the vertex/edge graph $G(P)=(\mathcal{V}(P),E)$.
\end{thm}

Next, we augment $G(P)$ by assigning a label to each vertex, where 
the label is an easily-computed invariant under unimodular affine transformations. To that end, we introduce the following quantity.

For each vertex $\bm{v}\in \mathcal{V}(P)$, we obtain the set $\{\bm{v}_{1},...,\bm{v}_{m}\}\subset \mathcal{V}(P)$ of adjacent vertices via graph $G(P)$. Subtracting $\bm{v}$ from each adjacent vertex yields the vectors $\bm{v}_{1}-\bm{v},...,\bm{v}_{m}-\bm{v}$. These $m$ $(m\geqslant n)$ vectors are integral and their rank is $n$. This is because $P$ is a full dimensional lattice polytope in $n$-dimensional space. We then consider the $n\times n$ matrix

\begin{equation}\label{ld}
	A_{\bm{v}}=(\bm{v}_{1}-\bm{v},...,\bm{v}_{m}-\bm{v})\cdot \begin{pmatrix}
		\bm{v}_{1}^{t}-\bm{v}^{t}\\ 
		\vdots\\ 
		\bm{v}_{m}^{t}-\bm{v}^{t}
	\end{pmatrix},
\end{equation}
which is an $n\times n$ positive definite integer symmetric matrix. Thus, for each vertex $\bm{v}\in \mathcal{V}(P)$, the determinant $|A_{\bm{v}}|$ is positive.

If $n$-dimensional lattice polytopes $P,P'\subset \mathbb{Z}^{n}$ are unimodularly isomorphic, there exists a unimodular affine transformation $\mathcal{U}=(U,\bm{Z})\in \mathcal{GL}_n(\mathbb{Z})\times \mathbb{Z}^{n}$ such that $P'=\mathcal{U}(P)=UP+\bm{Z}$. This transformation preserves the relative positional relationships between points and is invertible, with the inverse transformation also being a unimodular affine transformation, i.e., 
 $P=\mathcal{U}^{-1}(P')=U^{-1}P'-U^{-1}\bm{Z}$.

Thus, $\bm{v}$ is a vertex of $P$ with adjacent vertices $\{\bm{v}_{1},...,\bm{v}_{m}\}$ if and only if $\mathcal{U}(\bm{v})$ is a vertex of $P'$ with adjacent vertices $\{\mathcal{U}(\bm{v}_{1}),...,\mathcal{U}(\bm{v}_{m})\}$.  For the vertex $\mathcal{U}(\bm{v})$, consider the $n\times n$ matrix
\[
\begin{aligned}
	A_{\mathcal{U}(\bm{v})}&=(\mathcal{U}(\bm{v}_{1})-\mathcal{U}(\bm{v}),...,\mathcal{U}(\bm{v}_{m})-\mathcal{U}(\bm{v}))M\cdot M^{t} \begin{pmatrix}
		\mathcal{U}(\bm{v}_{1})^{t}-\mathcal{U}(\bm{v})^{t}\\ 
		\vdots\\ 
		(\mathcal{U}(\bm{v}_{m}))^{t}-(\mathcal{U}(\bm{v}))^{t}
	\end{pmatrix}\\
	&=U(\bm{v}_{1}-\bm{v},...,\bm{v}_{m}-\bm{v})\cdot \begin{pmatrix}
		\bm{v}_{1}^{t}-\bm{v}^{t}\\ 
		\vdots\\ 
		\bm{v}_{m}^{t}-\bm{v}^{t}
	\end{pmatrix}U^{t},
\end{aligned}
\]
where $M$ is an $m \times m$ permutation matrix. Since $\det(U)=\pm 1$, the determinants $|A_{\bm{v}}|$ and $|A_{\mathcal{U}(\bm{v})}|$ corresponding to $\bm{v}$ and $\mathcal{U}(\bm{v})$ are equal. This invariance motivates the following definition.

\begin{defn}\label{label1}
	Let $P\subset \mathbb{Z}^{n}$ be an $n$-dimensional lattice polytope with vertex set $\mathcal{V}(P)$. For each vertex $\bm{v}\in \mathcal{V}(P)$, the \emph{label} $\ell(\bm{v})$ of $\bm{v}$ is the determinant of the matrix $A_{\bm{v}}$ defined in (\ref{ld}).
\end{defn}
From the preceding analysis, this vertex label is invariant under unimodular affine transformations.
\begin{defn}
	Let $P\subset \mathbb{Z}^{n}$ be an $n$-dimensional lattice polytope with vertex/edge graph $G(P)$. For each vertex $\bm{v}$ of $G(P)$, we assign the label $\ell(\bm{v})$, yielding the labeled graph $\mathcal{G}(P)$.
\end{defn}
Further,  we extend the vertex labels to edges of 
 $\mathcal{G}(P)$: for each edge $\{\bm{v}_i,\bm{v}_j\}$, we define the edge label as $\ell(\bm{v}_i)+\ell(\bm{v}_j)$, which serves as a weight. We denote the resulting vertex- and edge-labeled graph by $\mathcal{GW}(P)$.
By the invariance of labels under unimodular affine transformations, we immediately obtain the following lemma.
\begin{lem}
	If $n$-dimensional lattice polytopes $P,P'\subset \mathbb{Z}^{n}$ are unimodularly isomorphic, the labeled graphs $\mathcal{GW}(P)$ and $\mathcal{GW}(P')$ are label-preserving isomorphic under unimodular affine transformations.
\end{lem}
Next, we study the correspondence of spanning trees in label-preserving isomorphic graphs.
\begin{lem}\label{11}
	If $n$-dimensional lattice polytopes $P,P'\subset \mathbb{Z}^{n}$ are unimodularly isomorphic, there is a unimodular affine transformation $\mathcal{U}=(U,\bm{Z})$ such that $P'=\mathcal{U}(P)=UP+\bm{Z}$. Under $\mathcal{U}$, the spanning trees of labeled graphs $\mathcal{GW}(P)$ and $\mathcal{GW}(P')$ maintain a one-to-one correspondence that is label-preserving isomorphic.
\end{lem} 

\begin{proof}
	Let $T=(\mathcal{V}(P),E_T)$ be a spanning tree of $\mathcal{GW}(P)$. Since  $\mathcal{U}$ is a label-preserving graph isomorphism between $\mathcal{GW}(P)$ and $\mathcal{GW}(P')$, the image graph $T'=\mathcal{U}(T)=(\mathcal{U}(\mathcal{V}(P)),\mathcal{U}(E_T))$ satisfies: 
	\begin{itemize}
		\item $\mathcal{U}(\mathcal{V}(P))=\mathcal{V}(P')$  (bijectivity of unimodular transformations),
		\item $\mathcal{U}(E_T)=\{\{\mathcal{U}(\bm{v}),\mathcal{U}(\bm{w})\}\ |\ \{\bm{v},\bm{w}\}\in E_T\}$.
	\end{itemize}
By the label invariance under unimodular transformations, vertex labels satisfy $\ell(\bm{v})=\ell(\mathcal{U}(\bm{v}))$,  hence edge labels satisfy $\ell(\bm{v})+\ell(\bm{w})=\ell(\mathcal{U}(\bm{v}))+\ell(\mathcal{U}(\bm{w}))$.
Moreover, $T'$  is a spanning tree: connectivity and acyclicity are preserved under graph isomorphism, so $T'$ spans $\mathcal{V}(P')$ and contains no cycles.
	
Conversely, let $T'=(\mathcal{V}(P'),E'_{T'})$ be a spanning tree of $\mathcal{GW}(P')$. The inverse transformation  $\mathcal{U}^{-1}=(U^{-1},-U^{-1}\bm{Z})$ is also a unimodular affine transformation and a label-preserving graph isomorphism. By analogous reasoning, $T=\mathcal{U}^{-1}(T')$ is a spanning tree of $\mathcal{GW}(P)$ with matching vertex and edge labels.
	
Thus, $\{T_1,...,T_t\}$ is the set of all spanning trees of $\mathcal{GW}(P)$ if and only if $\{\mathcal{U}(T_1),...,\mathcal{U}(T_t)\}$ is the set for 
 $\mathcal{GW}(P')$, and each $T_i$ is label-preserving isomorphic to $\mathcal{U}(T_i)$ under $\mathcal{U}$. 
\end{proof} 
Let $T$ and $T'$ be spanning trees of $\mathcal{GW}(P)$ and $\mathcal{GW}(P')$, respectively. If $T$ and $T'$ are not label-preserving isomorphic, they cannot be isomorphic under any unimodular affine transformation. This implies that unimodular affine transformations between polytopes can be verified by the existence of label-preserving isomorphic trees. Additionally, if the sum of edge weights (i.e., edge labels defined as $\ell(\bm{v}_i)+\ell(\bm{v}_j)$) in $T$ and $T'$ differs, $T$ and $T'$ cannot be label-preserving isomorphic. For this reason, we can focus on minimum spanning trees (MSTs) of $\mathcal{GW}(P)$ and $\mathcal{GW}(P')$ that are label-preserving isomorphic, identifying unimodular affine transformations through the label-preserving isomorphism of such MSTs.
\begin{rem}\label{vl}
 Since edge weights are defined as the sum of vertex labels, a vertex-label-preserving isomorphism naturally preserves edge labels, reducing the MST analysis to vertex labels alone. 
\end{rem}
To verify whether two trees 
$T$ and $T'$ with vertex labels $\ell(\bm{v})$ are label-preserving isomorphic, we employ the algorithm proposed in \cite{Kobler}, which efficiently checks isomorphism while respecting vertex labels. We describe the algorithm as the following Algorithm 3, which presents an efficient check for label-preserving isomorphism in trees by iteratively assigning new labels and comparing multisets of vertex labels.

\vskip 0.2cm
\begin{algorithm}[htb]
	\SetAlgoLined
	\KwIn{Two trees $T_1$ and $T_2$ with vertex labels $\ell(\bm{v})$.}
	\KwOut{"Label-preserving isomorphic" or "Non label-preserving isomorphic".}
	
	Label all leaf vertices of $T_1$ and $T_2$ with new label $1$\;
	$L_1 \gets \{(\ell(\bm{v}), 1) \mid \bm{v} \text{ is a leaf of } T_1\}$\;
	$L_2 \gets \{(\ell(\bm{v}), 1) \mid \bm{v} \text{ is a leaf of } T_2\}$\;
	\If{$L_1$ and $L_2$ are not equal as multisets}{
		\Return "Non label-preserving isomorphic"\;
	}
	
	\While{there exist vertices without new labels}{
		For each tree, define $S_i$ as the set of vertices without new labels where all but at most one neighbor have new labels\;
		
		\ForAll{$\bm{v} \in S_1 \cup S_2$}{
			Collect new labels of labeled neighbors, sort in non-decreasing order to form list $N(\bm{v})$\;
			Tentative label: label $\bm{v}$ with $T(\bm{v}) \gets (\ell(\bm{v}),N(\bm{v}))$ \;
		}
		
		$T_1^{\text{tent}} \gets \{T(\bm{v}) \mid \bm{v} \in S_1\}$\;
		$T_2^{\text{tent}} \gets \{T(\bm{v}) \mid \bm{v} \in S_2\}$\;
		\If{$T_1^{\text{tent}}$ and $T_2^{\text{tent}}$ are not equal as multisets}{
			\Return "Non label-preserving isomorphic"\;
		}
		
		Sort the elements of the sets $\{N(\bm{v}) \mid \bm{v} \in S_1\}$ and $\{N(\bm{v}) \mid \bm{v} \in S_2\}$ in non-decreasing order respectively \;
		In each tree, for the elements of $\{N(\bm{v}) \mid \bm{v} \in S_i\}$, the smallest $N(\bm{v})$ get the new label $l$, where $l$ be the smallest integer that has not been used as a new label before, then $l+1$ to the next smallest, and so on\;
	}
	
	\Return "Label-preserving isomorphic"\;
	\caption{Label-Preserving Isomorphism Check for Trees (LPT).}
	\label{alg:lpt}
\end{algorithm}

\vskip 0.2cm

	To illustrate the algorithm, consider the example in Figure \ref{lizi}, where $T_1$ and $T_2$ are labeled trees with vertex labels $\ell_i$ and $\ell'_i$ (corresponding to $\ell(\bm{v}_i)$ and $\ell(\bm{v}'_i)$ for $i = 1, \dots, 6$). The algorithm proceeds as follows:
	
	\begin{enumerate}
		\item \textbf{Step 1: Leaf Label Initialization} \\
		New label $1$ is assigned to all leaf vertices, forming the multisets:
		\[
		L_1 = \{(\ell_1, 1), (\ell_5, 1), (\ell_6, 1)\}, \quad L_2 = \{(\ell'_1, 1), (\ell'_5, 1), (\ell'_6, 1)\}.
		\]
		If $L_1$ and $L_2$ are equal, proceed to Step 2.
		
		\item \textbf{Step 2: Tentative Label Calculation} \\
		For vertices in $S_1$ and $S_2$ (with at most one unlabeled neighbor):
		\begin{itemize}
			\item For $T_1$: $S_1=\{\bm{v}_2, \bm{v}_4\}$. $\bm{v}_2$ has one labeled neighbor (label $1$), so $N(\bm{v}_2) = [1]$, tentative label $T(\bm{v}_2) = ( \ell_2, [1])$. $\bm{v}_4$ has two labeled neighbors (labels $1$), so $N(\bm{v}_4) = [1, 1]$, tentative label $T(\bm{v}_4) = (\ell_4, [1,1])$.
			\item For $T_2$: $S_2=\{\bm{v}'_2, \bm{v}'_4\}$. $\bm{v}'_4$ has one labeled neighbor (label $1$), so $N(\bm{v}'_4) = [1]$, tentative label $T(\bm{v}'_4) = ( \ell'_4, [1])$. $\bm{v}'_2$ has two labeled neighbors (labels $1$), so $N(\bm{v}'_2) = [1, 1]$, tentative label $T(\bm{v}'_2) = (\ell'_2, [1,1])$.
		\end{itemize}
		The tentative label multisets are:
		\[
		T_1^{\text{tent}} = \{(\ell_2, [1]), (\ell_4, [1, 1])\}, \quad T_2^{\text{tent}} = \{(\ell'_2, [1, 1]), (\ell'_4, [1])\}.
		\]
		If $T_1^{\text{tent}}$ and $T_2^{\text{tent}}$ are equal, proceed to Step 3.
		
		\item \textbf{Step 3: New Label Assignment} \\
	  In each tree, sort the elements of $\{N(\bm{v}) \mid \bm{v} \in S_i\}$ in non-decreasing order, and the smallest unused integer $l$ is assigned:
		\begin{itemize}
			\item For $T_1$: $N(\bm{v}_2)$ gets new label $2$ and $N(\bm{v}_4)$ gets new label $3$.
			\item For $T_2$: $N(\bm{v}'_4)$ gets new label $2$ and $N(\bm{v}'_2)$ gets new label $3$.
		\end{itemize}
		
		\item \textbf{Iteration (Steps 2--3 Repeat)} \\
		For vertices that have not been assigned new labels, such as $\bm{v}_3$ in $T_1$ and $\bm{v}'_3$ in $T_2$, we repeat Steps 2 and 3.
		\begin{itemize}
			\item $S_1=\{\bm{v}_3\}$ and $S_2=\{\bm{v}'_3\}$; $N(\bm{v}_3) = [2, 3]$ and $N(\bm{v}'_3) = [2, 3]$.
			\item $N(\bm{v}_3)$ and $N(\bm{v}'_3)$ get new label $4$ if $T_1^{\text{tent}} \equiv T_2^{\text{tent}}$.
		\end{itemize}
		
		\item \textbf{Termination} \\
		- All vertices assigned new labels and all multisets match $\Rightarrow$ "label-preserving isomorphic". \\
		- Any multiset mismatch $\Rightarrow$ "non label-preserving isomorphic".
	\end{enumerate}

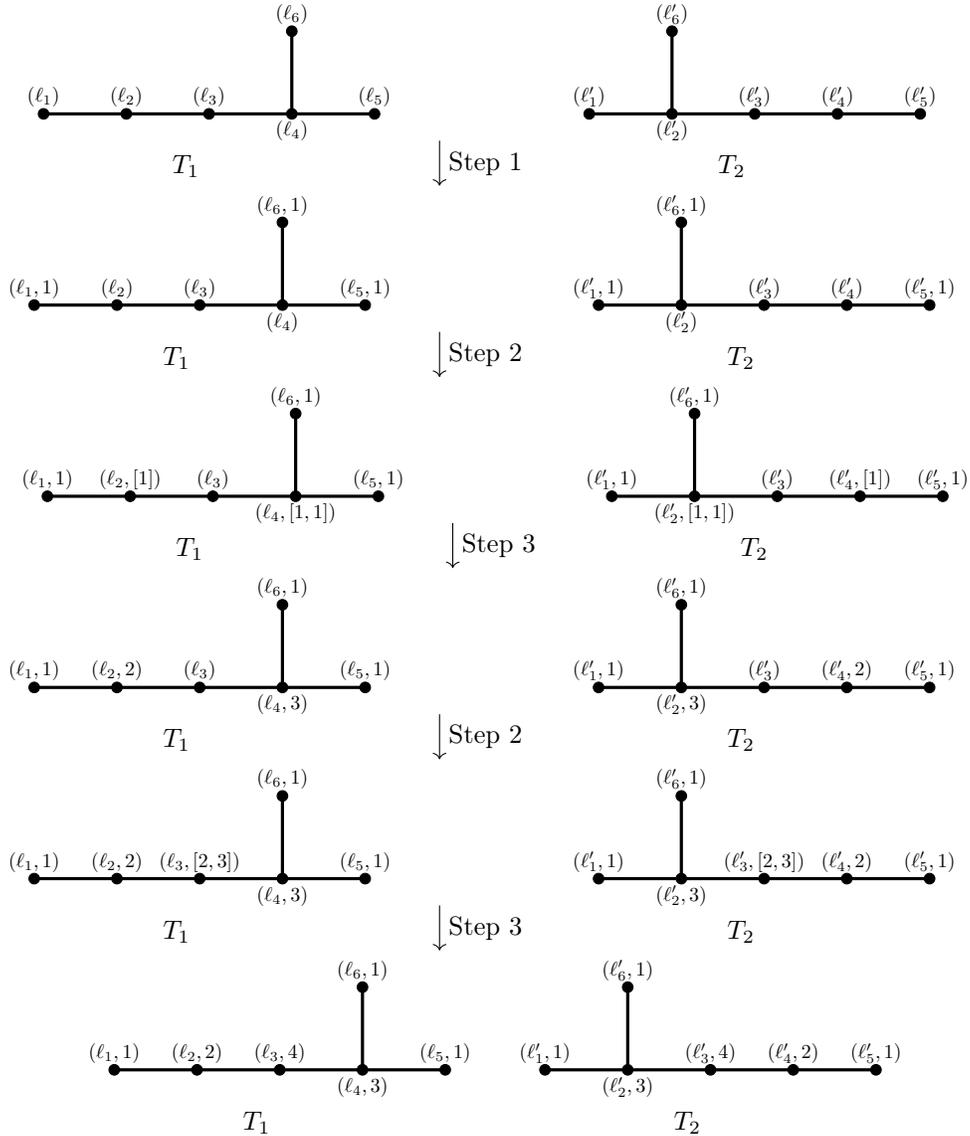
\begin{figure}[hbtp]
	\centering
	\begin{tikzpicture}                                                        
		\draw[black, very thick] (0,0)--(1.1,0);
		\draw[black, very thick] (1.1,0)--(2.2,0);
		\draw[black, very thick] (2.2,0)--(3.3,0);
		\draw[black, very thick] (3.3,0)--(4.4,0);
		\draw[black, very thick] (3.3,0)--(3.3,1.1);
		\filldraw [black] (0,0) circle (2pt) node [above, scale=0.75pt] {$(\ell_{1})$} ;
		\filldraw[black] (1.1,0) circle (2pt) node [above, scale=0.75pt] {$(\ell_{2})$};
		\filldraw[black] (2.2,0) circle (2pt) node [above, scale=0.75pt] {$(\ell_{3})$};
		\filldraw[black] (3.3,0) circle (2pt) node [below, scale=0.75pt] {$(\ell_{4})$};
		\filldraw[black] (4.4,0) circle (2pt) node [above, scale=0.75pt] {$(\ell_{5})$};
		\filldraw[black] (3.3,1.1) circle (2pt) node [above, scale=0.75pt] {$(\ell_{6})$};
		\node[left] at (2.2,-0.7){$T_1$};
	\end{tikzpicture}
	\quad
	\begin{tikzpicture}
		\draw [->] (-1,-2)--(-1,-2.6) node [midway, right] {Step 1};
	\end{tikzpicture}
	\quad
	\begin{tikzpicture}
		\draw[black, very thick] (0,0)--(1.1,0);
		\draw[black, very thick] (1.1,0)--(2.2,0);
		\draw[black, very thick] (2.2,0)--(3.3,0);
		\draw[black, very thick] (3.3,0)--(4.4,0);
		\draw[black, very thick] (1.1,0)--(1.1,1.1);
		\filldraw [black] (0,0) circle (2pt) node [above, scale=0.75pt] {$(\ell'_{1})$} ;
		\filldraw[black] (1.1,0) circle (2pt) node [below, scale=0.75pt] {$(\ell'_{2})$};
		\filldraw[black] (2.2,0) circle (2pt) node [above, scale=0.75pt] {$(\ell'_{3})$};
		\filldraw[black] (3.3,0) circle (2pt) node [above, scale=0.75pt] {$(\ell'_{4})$};
		\filldraw[black] (4.4,0) circle (2pt) node [above, scale=0.75pt] {$(\ell'_{5})$};
		\filldraw[black] (1.1,1.1) circle (2pt) node [above, scale=0.75pt] {$(\ell'_{6})$};
		\node[left] at (2.2,-0.7){$T_2$};
		
	\end{tikzpicture}
	\quad
	\begin{tikzpicture}                                                     
		\draw[black, very thick] (0,-1)--(1.1,-1);
		\draw[black, very thick] (1.1,-1)--(2.2,-1);
		\draw[black, very thick] (2.2,-1)--(3.3,-1);
		\draw[black, very thick] (3.3,-1)--(4.4,-1);
		\draw[black, very thick] (3.3,-1)--(3.3,0.1);
		\filldraw [black] (0,-1) circle (2pt) node [above, scale=0.75pt] {$(\ell_{1},1)$} ;
		\filldraw[black] (1.1,-1) circle (2pt) node [above, scale=0.75pt] {$(\ell_{2})$};
		\filldraw[black] (2.2,-1) circle (2pt) node [above, scale=0.75pt] {$(\ell_{3})$};
		\filldraw[black] (3.3,-1) circle (2pt) node [below, scale=0.75pt] {$(\ell_{4})$};
		\filldraw[black] (4.4,-1) circle (2pt) node [above, scale=0.75pt] {$(\ell_{5},1)$};
		\filldraw[black] (3.3,0.1) circle (2pt) node [above, scale=0.75pt] {$(\ell_{6},1)$};
		\node[left] at (2.2,-1.7){$T_1$};
	\end{tikzpicture}
	\quad
	\begin{tikzpicture}
		\draw [->] (-1,-3)--(-1,-3.6)node [midway, right] {Step 2};
	\end{tikzpicture}
	\quad
	\begin{tikzpicture}
		\draw[black, very thick] (0,0)--(1.1,0);
		\draw[black, very thick] (1.1,0)--(2.2,0);
		\draw[black, very thick] (2.2,0)--(3.3,0);
		\draw[black, very thick] (3.3,0)--(4.4,0);
		\draw[black, very thick] (1.1,0)--(1.1,1.1);
		\filldraw [black] (0,0) circle (2pt) node [above, scale=0.75pt] {$(\ell'_{1},1)$} ;
		\filldraw[black] (1.1,0) circle (2pt) node [below, scale=0.75pt] {$(\ell'_{2})$};
		\filldraw[black] (2.2,0) circle (2pt) node [above, scale=0.75pt] {$(\ell'_{3})$};
		\filldraw[black] (3.3,0) circle (2pt) node [above, scale=0.75pt] {$(\ell'_{4})$};
		\filldraw[black] (4.4,0) circle (2pt) node [above, scale=0.75pt] {$(\ell'_{5},1)$};
		\filldraw[black] (1.1,1.1) circle (2pt) node [above, scale=0.75pt] {$(\ell'_{6},1)$};
		\node[left] at (2.2,-0.7){$T_2$};
	\end{tikzpicture}
	
	\quad
	\begin{tikzpicture}                                                     
		\draw[black, very thick] (0,-1)--(1.1,-1);
		\draw[black, very thick] (1.1,-1)--(2.2,-1);
		\draw[black, very thick] (2.2,-1)--(3.3,-1);
		\draw[black, very thick] (3.3,-1)--(4.4,-1);
		\draw[black, very thick] (3.3,-1)--(3.3,0.1);
		\filldraw [black] (0,-1) circle (2pt) node [above, scale=0.75pt] {$(\ell_{1},1)$} ;
		\filldraw[black] (1.1,-1) circle (2pt) node [above, scale=0.75pt] {$(\ell_{2},[1])$};
		\filldraw[black] (2.2,-1) circle (2pt) node [above, scale=0.75pt] {$(\ell_{3})$};
		\filldraw[black] (3.3,-1) circle (2pt) node [below, scale=0.75pt] {$(\ell_{4},[1,1])$};
		\filldraw[black] (4.4,-1) circle (2pt) node [above, scale=0.75pt] {$(\ell_{5},1)$};
		\filldraw[black] (3.3,0.1) circle (2pt) node [above, scale=0.75pt] {$(\ell_{6},1)$};
		\node[left] at (2.2,-1.7){$T_1$};
	\end{tikzpicture}
	\quad
	\begin{tikzpicture}
		\draw [->] (-1,-3)--(-1,-3.6) node [midway, right] {Step 3};
	\end{tikzpicture}
	\quad
	\begin{tikzpicture}
		\draw[black, very thick] (0,0)--(1.1,0);
		\draw[black, very thick] (1.1,0)--(2.2,0);
		\draw[black, very thick] (2.2,0)--(3.3,0);
		\draw[black, very thick] (3.3,0)--(4.4,0);
		\draw[black, very thick] (1.1,0)--(1.1,1.1);
		\filldraw [black] (0,0) circle (2pt) node [above, scale=0.75pt] {$(\ell'_{1},1)$} ;
		\filldraw[black] (1.1,0) circle (2pt) node [below, scale=0.75pt] {$(\ell'_{2},[1,1])$};
		\filldraw[black] (2.2,0) circle (2pt) node [above, scale=0.75pt] {$(\ell'_{3})$};
		\filldraw[black] (3.3,0) circle (2pt) node [above, scale=0.75pt] {$(\ell'_{4},[1])$};
		\filldraw[black] (4.4,0) circle (2pt) node [above, scale=0.75pt] {$(\ell'_{5},1)$};
		\filldraw[black] (1.1,1.1) circle (2pt) node [above, scale=0.75pt] {$(\ell'_{6},1)$};
		\node[left] at (2.2,-0.7){$T_2$};
	\end{tikzpicture}
	\quad
	\begin{tikzpicture}                                                     
		\draw[black, very thick] (0,-1)--(1.1,-1);
		\draw[black, very thick] (1.1,-1)--(2.2,-1);
		\draw[black, very thick] (2.2,-1)--(3.3,-1);
		\draw[black, very thick] (3.3,-1)--(4.4,-1);
		\draw[black, very thick] (3.3,-1)--(3.3,0.1);
		\filldraw [black] (0,-1) circle (2pt) node [above, scale=0.75pt] {$(\ell_{1},1)$} ;
		\filldraw[black] (1.1,-1) circle (2pt) node [above, scale=0.75pt] {$(\ell_{2},2)$};
		\filldraw[black] (2.2,-1) circle (2pt) node [above, scale=0.75pt] {$(\ell_{3})$};
		\filldraw[black] (3.3,-1) circle (2pt) node [below, scale=0.75pt] {$(\ell_{4},3)$};
		\filldraw[black] (4.4,-1) circle (2pt) node [above, scale=0.75pt] {$(\ell_{5},1)$};
		\filldraw[black] (3.3,0.1) circle (2pt) node [above, scale=0.75pt] {$(\ell_{6},1)$};
		\node[left] at (2.2,-1.7){$T_1$};
	\end{tikzpicture}
	\quad
	\begin{tikzpicture}
		\draw [->] (-1,-3)--(-1,-3.6) node [midway, right] {Step 2};
	\end{tikzpicture}
	\quad
	\begin{tikzpicture}
		\draw[black, very thick] (0,0)--(1.1,0);
		\draw[black, very thick] (1.1,0)--(2.2,0);
		\draw[black, very thick] (2.2,0)--(3.3,0);
		\draw[black, very thick] (3.3,0)--(4.4,0);
		\draw[black, very thick] (1.1,0)--(1.1,1.1);
		\filldraw [black] (0,0) circle (2pt) node [above, scale=0.75pt] {$(\ell'_{1},1)$} ;
		\filldraw[black] (1.1,0) circle (2pt) node [below, scale=0.75pt] {$(\ell'_{2},3)$};
		\filldraw[black] (2.2,0) circle (2pt) node [above, scale=0.75pt] {$(\ell'_{3})$};
		\filldraw[black] (3.3,0) circle (2pt) node [above, scale=0.75pt] {$(\ell'_{4},2)$};
		\filldraw[black] (4.4,0) circle (2pt) node [above, scale=0.75pt] {$(\ell'_{5},1)$};
		\filldraw[black] (1.1,1.1) circle (2pt) node [above, scale=0.75pt] {$(\ell'_{6},1)$};
		\node[left] at (2.2,-0.7){$T_2$};
	\end{tikzpicture}
	\quad
	\begin{tikzpicture}                                                     
		\draw[black, very thick] (0,0)--(1.1,0);
		\draw[black, very thick] (1.1,0)--(2.2,0);
		\draw[black, very thick] (2.2,0)--(3.3,0);
		\draw[black, very thick] (3.3,0)--(4.4,0);
		\draw[black, very thick] (3.3,0)--(3.3,1.1);
		\filldraw [black] (0,0) circle (2pt) node [above, scale=0.75pt] {$(\ell_{1},1)$} ;
		\filldraw[black] (1.1,0) circle (2pt) node [above, scale=0.75pt] {$(\ell_{2},2)$};
		\filldraw[black] (2.2,0) circle (2pt) node [above, scale=0.75pt] {$(\ell_{3},[2,3])$};
		\filldraw[black] (3.3,0) circle (2pt) node [below, scale=0.75pt] {$(\ell_{4},3)$};
		\filldraw[black] (4.4,0) circle (2pt) node [above, scale=0.75pt] {$(\ell_{5},1)$};
		\filldraw[black] (3.3,1.1) circle (2pt) node [above, scale=0.75pt] {$(\ell_{6},1)$};
		\node[left] at (2.2,-0.7){$T_1$};
	\end{tikzpicture}
	\quad
	\begin{tikzpicture}
		\draw [->] (-1,-3)--(-1,-3.6) node [midway, right] {Step 3};
	\end{tikzpicture}
	\quad
	\begin{tikzpicture}
		\draw[black, very thick] (0,0)--(1.1,0);
		\draw[black, very thick] (1.1,0)--(2.2,0);
		\draw[black, very thick] (2.2,0)--(3.3,0);
		\draw[black, very thick] (3.3,0)--(4.4,0);
		\draw[black, very thick] (1.1,0)--(1.1,1.1);
		\filldraw [black] (0,0) circle (2pt) node [above, scale=0.75pt] {$(\ell'_{1},1)$} ;
		\filldraw[black] (1.1,0) circle (2pt) node [below, scale=0.75pt] {$(\ell'_{2},3)$};
		\filldraw[black] (2.2,0) circle (2pt) node [above, scale=0.75pt] {$(\ell'_{3},[2,3])$};
		\filldraw[black] (3.3,0) circle (2pt) node [above, scale=0.75pt] {$(\ell'_{4},2)$};
		\filldraw[black] (4.4,0) circle (2pt) node [above, scale=0.75pt] {$(\ell'_{5},1)$};
		\filldraw[black] (1.1,1.1) circle (2pt) node [above, scale=0.75pt] {$(\ell'_{6},1)$};
		\node[left] at (2.2,-0.7){$T_2$};
	\end{tikzpicture}
	
	\quad
	\begin{tikzpicture}                                                     
		\draw[black, very thick] (0,0)--(1.1,0);
		\draw[black, very thick] (1.1,0)--(2.2,0);
		\draw[black, very thick] (2.2,0)--(3.3,0);
		\draw[black, very thick] (3.3,0)--(4.4,0);
		\draw[black, very thick] (3.3,0)--(3.3,1.1);
		\filldraw [black] (0,0) circle (2pt) node [above, scale=0.75pt] {$(\ell_{1},1)$} ;
		\filldraw[black] (1.1,0) circle (2pt) node [above, scale=0.75pt] {$(\ell_{2},2)$};
		\filldraw[black] (2.2,0) circle (2pt) node [above, scale=0.75pt] {$(\ell_{3},4)$};
		\filldraw[black] (3.3,0) circle (2pt) node [below, scale=0.75pt] {$(\ell_{4},3)$};
		\filldraw[black] (4.4,0) circle (2pt) node [above, scale=0.75pt] {$(\ell_{5},1)$};
		\filldraw[black] (3.3,1.1) circle (2pt) node [above, scale=0.75pt] {$(\ell_{6},1)$};
		\node[left] at (2.2,-0.7){$T_1$};
	\end{tikzpicture}
	\quad
	\begin{tikzpicture}
		\draw[black, very thick] (0,0)--(1.1,0);
		\draw[black, very thick] (1.1,0)--(2.2,0);
		\draw[black, very thick] (2.2,0)--(3.3,0);
		\draw[black, very thick] (3.3,0)--(4.4,0);
		\draw[black, very thick] (1.1,0)--(1.1,1.1);
		\filldraw [black] (0,0) circle (2pt) node [above, scale=0.75pt] {$(\ell'_{1},1)$} ;
		\filldraw[black] (1.1,0) circle (2pt) node [below, scale=0.75pt] {$(\ell'_{2},3)$};
		\filldraw[black] (2.2,0) circle (2pt) node [above, scale=0.75pt] {$(\ell'_{3},4)$};
		\filldraw[black] (3.3,0) circle (2pt) node [above, scale=0.75pt] {$(\ell'_{4},2)$};
		\filldraw[black] (4.4,0) circle (2pt) node [above, scale=0.75pt] {$(\ell'_{5},1)$};
		\filldraw[black] (1.1,1.1) circle (2pt) node [above, scale=0.75pt] {$(\ell'_{6},1)$};
		\node[left] at (2.2,-0.7){$T_2$};
	\end{tikzpicture}
	\centering
	\caption{ Label-Preserving Isomorphism Check for Trees Using LPT Algorithm}\label{lizi}
\end{figure}

Next, given $n$-dimensional convex lattice polytopes $P,P'\subset \mathbb{Z}^{n}$ with $d$ vertices, we describe our algorithm for computing all unimodular affine transformation between $P$ and $P'$.
\begin{itemize}
	\item[\textbf{Step\ 1.}] According to Theorem \ref{veg}, we can get the vertex/edge graphs $G(P)$ and $G(P')$. By (\ref{ld}), for each vertex of graphs, we can get the corresponding label, so as to get the labeled graphs $\mathcal{GW}(P)$ and $\mathcal{GW}(P')$.
	\item[\textbf{Step\ 2.}] Compute a minimum spanning tree $T$ of $\mathcal{GW}(P)$ and all minimum spanning trees of $\mathcal{GW}(P')$, denoted by  $T'_{min}=\{T'_{1},...,T'_{t}\}$. Remove edge weights from $T$ and $T'_j$ to obtain labeled spanning trees $\mathcal{T}$ and $\mathcal{T}_j$ (retaining vertex labels), where $j=1,..,t$. Compute the label-preserving automorphism group $\mathrm{Aut}_{\ell}(\mathcal{T})$ of $\mathcal{T}$. 

	\item[\textbf{Step\ 3.}]  We use Algorithm LPT to search spanning trees in  $\{\mathcal{T}'_{1},...,\mathcal{T}'_{t}\}$ that are label-preserving isomorphic to $\mathcal{T}$. 
	\item[\textbf{Step\ 4.}] 
	Suppose $\varphi$ is a label-preserving isomorphism between $\mathcal{T}$ and $\mathcal{T}'_j$  ($j\in \{1,...,t\}$), inducing a vertex map  $\varphi$: $\mathcal{V}(P) \to \mathcal{V}(P')$. Two polytopes are unimodularly isomorphic iff such $\varphi$ exists, and any unimodular affine transformation $\Phi: \mathbb{Z}^{n}\to \mathbb{Z}^{n}$ mapping $P$ to $P'$ can be factored as $\varphi \circ \chi$, where $\chi \in  \mathrm{Aut}_{\ell}(\mathcal{T})$.
	
	It remains to decide whether a particular choice of $\chi \in  \mathrm{Aut}_{\ell}(\mathcal{T})$ determines a unimodular affine transformation $\varphi \circ \chi : \mathbb{Z}^{n}\to \mathbb{Z}^{n}$ sending $P$ to $P'$. For this,  select $n$ linearly independent vectors $\{\bm{w}_1,...,\bm{w}_n\} \subseteq \mathcal{V}(P)-b_{P}=\{\bm{v}_1-b_{P},...,\bm{v}_d-b_{P}\}$. For $\chi \in \mathrm{Aut}_{\ell}(\mathcal{T})$, we denote the image of $\bm{w}_i+b_{P}$ in $\mathcal{V}(P')$ by $\bm{w}'_{i}+b_{P'}$ under $\varphi \circ \chi$. Construct matrices $\mathcal{W}=(\bm{w}_1,...,\bm{w}_n)$ and $\mathcal{W}'=(\bm{w}'_1,...,\bm{w}'_n)$,  then check if $U:=\mathcal{W}'\mathcal{W}^{-1}$ is unimodular and $\{U\bm{v}|\bm{v}\in \mathcal{V}(P)-b_{P}\}=\mathcal{V}(P')-b_{P'}$. 
	
	\item[\textbf{Step\ 5.}] For each label-preserving isomorphism $\varphi$, and for each $\chi\in \mathrm{Aut}_{\ell}(\mathcal{T})$, we do the process in Step 4. 
\end{itemize}

\begin{thm}
	Through the above algorithm, all unimodular affine transformations between $P$ and $P'$ can be obtained.
\end{thm}
\begin{proof}
	Let $T$ be a minimum spanning tree of $\mathcal{GW}(P)$. For any automorphism $\mathcal{A}\in \mathrm{Aut}(P)$, $\mathcal{A}(T)$ is  either $T$ itself or another minimum spanning tree of $\mathcal{GW}(P)$. 
	
By Lemma \ref{11},	for any unimodular affine transformation $\mathcal{U}$ such that $P'=\mathcal{U}(P)$, the composition $\mathcal{U} \circ \mathcal{A}$ maps $T$ to an MST of $\mathcal{GW}(P')$.  Specifically, $\{\mathcal{U}(\mathcal{A}(T)): \mathcal{A}\in \mathrm{Aut}(P)\}\subseteq T'_{min}$, and each $\mathcal{U}(\mathcal{A}(T))$ is label-preserving isomorphic to $T$. By Remark \ref{vl}, for MSTs, we only consider isomorphisms that preserve vertex labels.

	Thus, By Algorithm LPT, we identify all spanning trees in $\{\mathcal{T}'_1,...,\mathcal{T}'_t\}$ that are label-preserving isomorphic to $\mathcal{T}$, corresponding to potential vertex maps $\varphi$.  Since every unimodular affine  transformation $\mathcal{UA}$ between $P$ and $P'$ can be decomposed as $\varphi \circ \chi$ with $\chi \in \mathrm{Aut}_{\ell}(\mathcal{T})$, iterating over all $\varphi$ and $\chi$ can find all such transformations.
\end{proof}

\noindent \emph{Advantages of the algorithm}.  Firstly, the vertex/edge graph $G(P)$ is constructed and the vertex labels $\ell(\bm{v}) = \det(A_{\bm{v}})$ are computed in polynomial time, ensuring the initial step (the construction of $\mathcal{GW}(P)$) scales as $\mathcal{O}(\text{poly}(d, n))$. This contrasts with exponential-time approaches for general polytope isomorphism (e.g., \cite{Grinis}),  thus making the algorithm feasible for practical applications.

Secondly, we reduce unimodular isomorphism problem of lattice polytopes to label-preserving isomorphism of minimum spanning trees (MSTs).
Focusing MSTs of $\mathcal{GW}(P)$ and $\mathcal{GW}(P')$ offers several key benefits: 
\begin{itemize}
  
	\item \textbf{Reduction of Search Spaces}: MSTs have the smallest possible total edge weight, reducing the number of trees (graphs) to check. In particular, the number of MSTs (\( t \)) for asymmetric polytopes is generally constant (\( t = \mathcal{O}(1) \)), avoiding the exponential growth of search spaces common in high-dimensional problems.  Additionally, we can efficiently construct an MST via Kruskal's algorithm in \( \mathcal{O}(d^2 \log d) \) time.

\item \textbf{Efficiency of Isomorphism Checking}:  For MSTs, we only consider isomorphisms that preserve vertex labels.
For labeled spanning trees $\mathcal{T}$ and $\mathcal{T}'$, isomorphism is checked using Kobler's algorithm \cite{Kobler}, which runs in $\mathcal{O}(d \log d)$ time for trees with $d$ vertices. This is inherently easier to check than general graphs.

\item \textbf{Reduction of Automorphism Group}: 
Computing the label-preserving automorphism group \( \text{Aut}_{\ell}(\mathcal{T}) \) (Step 2) uses tree centroid decomposition in \( \mathcal{O}(d) \) time. For highly symmetric high-dimensional polytopes, this efficiently enumerates all possible automorphic transformations, preventing redundant computations that would otherwise arise from explicit symmetry enumeration.  
\end{itemize}

\end{document}